\documentclass[11pt]{amsart}
\usepackage{fullpage}
\usepackage{amsmath,amssymb,amsfonts}
\usepackage{enumerate}
\usepackage{amscd, amssymb, latexsym, amsmath, amscd}
\usepackage{graphicx}
\usepackage{aurical}
\usepackage[T1]{fontenc}
\usepackage{mathrsfs}

\newtheorem{theorem}{Theorem}

\newtheorem{conjecture}[theorem]{Conjecture}
\newtheorem{corollary}[theorem]{Corollary}

\newtheorem{definition}[theorem]{Definition}

\newtheorem{lemma}[theorem]{Lemma}

\newtheorem{proposition}[theorem]{Proposition}

\newcommand{\Aut}{\operatorname{Aut}}
\newcommand{\Z}{\mathbb{Z}}

\newcommand{\R}{\mathbb{R}}
\newcommand{\C}{\mathbb{C}}

\newcommand{\Ind}{\operatorname{Ind}}
\newcommand{\Tr}{\operatorname{Tr}}
\newcommand{\Orb}{\mathcal{O}}
\renewcommand{\k}{\mbox{\Fontauri k}}

\title{A conjecture of Sakellaridis-Venkatesh \\on the unitary spectrum of spherical varieties.}
\author{Wee Teck Gan and Raul Gomez}
\address{Department of Mathematics, National University of Singapore, Block S17, 10 Lower Kent Ridge Road, Singapore 587628}
\email{matgwt@nus.edu.sg}
\email{matrgm@nus.edu.sg}

\dedicatory{to Nolan Wallach, \\
 with admiration and appreciation}
\begin{document}
\maketitle

\section{Introduction}
The spectral decomposition of the unitary representation $L^2(H \backslash G)$ when   $X= H\backslash G$ is a symmetric space has been studied extensively, especially in the case when $G$ is a real Lie group. In particular, through the work of many authors (such as \cite{FJ}, \cite{OM}, \cite{V}, \cite{De1} and \cite{BS}), one now has the full Plancherel theorem in this setting.
\vskip 5pt

In a recent preprint \cite{Sakellaridis-Venkatesh}, Sakellaridis and Venkatesh considered the more general setting where $X = H \backslash G$ is a spherical variety and $G$ is a real or p-adic group. Motivated by the study of periods in the theory of automorphic forms and the comparison of relative trace formulas, they formulated an approach to this problem  in the framework of Langlands functoriality. More precisely, led by and refining the work of Gaitsgory-Nadler \cite{GN} in the geometric Langlands program, they associated to a spherical variety $X= H \backslash G$ (satisfying some additional technical hypotheses) 
\begin{itemize}
\item a dual group $\check{G}_X$;
\item a natural map $\iota: \check{G}_X \times  SL_2(\C) \longrightarrow \check{G}$
\end{itemize}
The map $\iota$ induces a map from the set of tempered L-parameters of $G_X$ to the set of Arthur parameters of $G$, and if one is very optimistic, it may even give rise to a map
\[  \iota_*  :  \widehat{G}_X \longrightarrow  \widehat{G} \]
where $G_X$ is a (split) group with dual group $\check{G}_X$ and $\widehat{G}_X$ and $\widehat{G}$ refer to the unitary dual of the relevant groups. Assuming for simplicity that this is the case, one has the following conjecture:
\vskip 15pt

\noindent{\bf \underline{Sakellaridis-Venkatesh Conjecture}}
\vskip 5pt
 One has a spectral decomposition
 \[  L^2(H \backslash G) \cong  \int_{\widehat{G}_X}   W(\pi) \otimes \iota_*(\pi)   \, d\mu(\pi) \]
where $\mu$ is the Plancherel measure of $\widehat{G}_X$ and $W(\pi)$ is some multiplicity space.
\vskip 5pt

 In particular, the class of the spectral measure of $L^2(H\backslash G)$ is absolutely continuous with respect to that of the pushforward by $\iota_*$ of the Plancherel measure  on $\widehat{G}_X$, and its support is contained in those Arthur parameters of $G$ which factor through $\iota$. In addition,  one expects that the multiplicity space $W(\pi)$  is related to the space of continuous $H$-invariant functionals on the representation $\iota_*(\pi)$.

\vskip 10pt

The main purpose of this paper is to verify the above conjecture in many cases when $H \backslash G$, or equivalently $G_X$, has low rank, and to specify the multiplicity space $W(\pi)$.
In particular, we demonstrate this conjecture for many cases when $G_X$ has rank $1$, and also some cases when $G_X$ has rank $2$ or $3$ (see the tables in \cite[\S 15 and 
\S 16]{Sakellaridis-Venkatesh}).  More precisely, our main result is:
\vskip 10pt

\begin{theorem}
The conjecture of Sakellaridis-Venkatesh holds for the spherical varieties $H \backslash G$ listed in the following tables.
\vskip 5pt

\begin{table}[ht]
\label{table1}
\begin{center}
\begin{tabular}{|c|c|c|c|} 
\hline 
$H \backslash G$ & $GL_{n-1} \backslash GL_n$ & $SO_{n-1}\backslash SO_n$ & $Sp_{2n-2} \backslash Sp_{2n}$ \\
\hline
 $G_X$ & $GL_2$  & $\widetilde{SL}_2$ & $SO(4)$ \\
 \hline
 \end{tabular}
 \end{center}
 \caption{Classical cases}
\end{table}

\begin{table}[ht]
\begin{center}
\begin{tabular}{|c|c|c|c|c|c|c|c|c|}
\hline
$H \backslash G$ & $SO_3 \backslash SL_3$ & 
$Sp_6 \backslash SL_6$ & $SL_3 \backslash G_2$ & $(J,\psi)\ \backslash G_2$
& $G_2 \backslash Spin_7$ & $G_2 \backslash Spin_8$ & $Spin_9 \backslash F_4$ &
$F_4 \backslash E_6$  \\
\hline
 $G_X$  & $\widetilde{SL}_3$  &  $SL_3$ & $\tilde{SL}_2$ & $PGL_3$   & $SL_2$ & $SL_2^3/\Delta \mu_2$ & $PGL_2$ & $SL_3$ \\
\hline
\end{tabular}
\end{center}
\caption{Exceptional cases}
\end{table}
\end{theorem}
We refer the reader to the main body of the paper for the precise statements and unexplained notation.
\vskip 5pt

The theorem is proved using the technique of theta correspondence. 
More precisely, it turns out that for the groups listed in the above table, one has a reductive dual pair
\[  G_X \times G \subset S \]
for some larger group $S$. One then studies the restriction of the minimal representation of $S$ to the subgroup $G_X \times G$. In the context of theta correspondence in smooth representation theory,  one can typically show the following rough statement: 
\[  \text{\em A representation $\pi$ of $G$ has $\psi$-generic (and hence nonzero) theta lift to $G_X$} \]
\[  \Updownarrow \]
\[   \text{$\pi$ has nonzero $H$-period.} \]
Our main theorem is thus the $L^2$-manifestation of this phenomenon, giving a description of $L^2(H \backslash G)$ in terms of $L^2(G_X)$.  
\vskip 5pt

This idea is not really new: a well known example of this kind of result is the correspondence between the irreducible components of the spherical harmonics on $\R^{n}$ under the action of $O(n,\R)$, and holomorphic discrete series of the group $\widetilde{SL}(2,\R)$, the double cover of $SL(2,\R)$. Another example is given by the classical paper of Rallis and Schiffmann \cite{Rallis-Schiffman:Weil} where they used the oscillator representation to relate the discrete spectrum of $L^{2}(O(p,q-1)\backslash O(p,q))$ with the discrete series representations of $\widetilde{SL}(2,\R)$. Later, Howe \cite{Howe:someresults} showed how these results can be inferred from his general theory of reductive dual pairs, and essentially provided a description of the Plancherel measure of $L^{2}(O(p,q-1)\backslash O(p,q))$ in terms of the representation theory of $\widetilde{SL}(2,\R)$. Then \O{}rsted and Zhang \cite{Zhang-Orsted:L2II} proved a similar result for the space $L^{2}(U(p,q-1)\backslash U(p,q))$ in terms of the representation theory of $U(1,1)$. We give a more steamlined treatment of these classical cases in Section 2, which accounts for Table 1. The rest of the paper is then
 devoted to the exceptional cases listed in Table 2.
\vskip 5pt

\vskip 10pt

\noindent{\bf Acknowledgments:} Both authors would like to pay tribute to Nolan Wallach for  his guidance, encouragement and friendship over the past few years. It is an honor to be his colleague and student respectively. We wish him all the best in his retirement from UCSD, and hope to continue to interact with him mathematically and personally for many years to come. 
\vskip 10pt

The research of the first author is partially supported by NSF grant 0801071 and a startup grant from the National University of Singapore. 

  \section{\bf Classical Dual Pairs}

 We begin by introducing the classical dual pairs.
\vskip 10pt

\subsection{Division algebra $D$.}

Let $\k$ be a local field, and let $|\cdot|$ denote its absolute value. Let $D = \k$, a quadratic field extension of $\k$ or the quaternion division $\k$-algebra, and let $x \mapsto \overline{x}$ be its canonical involution. The case when $D$ is the split quadratic algebra or quaternion algebra can also be included in the discussion, but for simplicity, we shall stick with division algebras.  We have the \emph{trace} map $Tr(x) = x + \overline{x} \in \k$ and the \emph{norm} map $Q(x)  = x \cdot \overline{x} \in \k$. 
 \vskip 5pt

\subsection{\bf  Hermitian $D$-modules.}

Let $V$ and $W$ be two right $D$-modules. We will denote the set of right $D$-module morphisms between $V$ and $W$ by
\[
 Hom_{D}(V,W)=\{T:V\longrightarrow W \,| \, \mbox{$T(v_{1}a+v_{2}b)=T(v_{1})a+T(v_{2})b$ for all $v_{1}$, $v_{2} \in V$, $a$, $b\in D$}\}.
\]
In the same way, if $V$ and $W$ are two left $D$-modules, we set 
\[
 Hom_{D}(V,W)=\{T:V\longrightarrow W \,| \, \mbox{$(av_{1}+bv_{2})T=a(v_{1})T+b(v_{2})T$ for all $v_{1}$, $v_{2} \in V$, $a$, $b\in D$}\}.
\]
If $V=W$, we will denote this set by $End_{D}(V)$. Notice that for right $D$-module morphisms we are putting the argument on the right, while for left $D$-module morphisms we are putting it on the left.
\vskip 5pt

 In general, for every statement involving \emph{right} $D$-modules one can make an analogous one involving \emph{left} $D$-modules. From now on, we will focus on right $D$-modules, and we will let the reader with the task of making the corresponding definitions and statements involving left $D$-modules. Set
\[
 GL(V,D)=\{T\in End_{D}(V) \, | \, \mbox{T is invertible}\}.
\]
When it is clear from the context what the division algebra is, we will just denote this group  by $GL(V)$.
\vskip 5pt

 Let $V'$ be the set of right $D$-linear functionals on $V$. There is a natural left $D$-module structure on $V'$ given by setting 
\[  (a\lambda)(v)=a\lambda(v), \quad \text{for all $a\in D$, $v\in V$, and $\lambda \in V'$.} \]
Observe that with this structure, $W\otimes_{D}V'$ is naturally isomorphic to $Hom_{D}(V,W)$ as a $\k$-vector space. Given $T\in Hom_{D}(V,W)$, we will define an element in $Hom_{D}(W',V')$, which we will also denote $T$, by setting $(\lambda T)(v):=\lambda(Tv)$. This correspondence gives rise to natural isomorphisms between $End_{D}(V)$ and $End_{D}(V')$ and between $GL(V)$ and $GL(V')$.

\begin{definition}
Let $\varepsilon=\pm 1$. We say that $(V,B)$ is a right $\varepsilon$-Hermitian $D$-module, if $V$ is a right $D$-module and $B$ is an $\varepsilon$-Hermitian form, i.e $B:V\times V \longrightarrow D$ is a map such that
\begin{enumerate}
 \item $B$ is \emph{sesquilinear}. That is, for all $v_{1}$, $v_{2}$, $v_{3}\in V$, $a$, $b\in D$, 
  \[ B(v_{1},v_{2}a+v_{3}b)=B(v_{1},v_{2})a+B(v_{1},v_{3})b \quad \text{and} \quad B(v_{1}a + v_{2}b,v_{3})=\overline{a}B(v_{1},v_{3})+\overline{b}B(v_{2},v_{3}).\]
\item $B$ is $\varepsilon$-\emph{Hermitian}. That is, 
$$
B(v,w)=\varepsilon\overline{B(w,v)} \qquad \mbox{ for all $v,w\in V$.}
$$
\item $B$ is \emph{non-degenerate}. 
\end{enumerate} 
\end{definition}
To define left $\varepsilon$-Hermitian $D$-modules $(V,B)$, we just have to replace the sesquilinear condition by
$$B(av_{1}+bv_{2},v_{3})=aB(v_{1},v_{3})+bB(v_{2},v_{3}) \quad \text{and} \quad B(v_{1},av_{2} + bv_{3})=B(v_{1},v_{2})\overline{a}+B(v_{1},v_{3})\overline{b},$$
for all $v_{1}$, $v_{2}$, $v_{3}\in V$, $a$, $b\in D$.

Given a right $\varepsilon$-Hermitian $D$-module $(V,B)$, we will define
\[
G(V,B)=\{g\in GL(V) \, | \, \mbox{$B(g v,g w)=B(v,w)$ for all $v$, $w\in V$}\},
\]
to be the subgroup of $GL(V)$ preserving the $\varepsilon$-Hermitian form $B$. When there is no risk of confusion regarding $B$, we will denote this group just by $G(V)$. Usually, $1$-Hermitian $D$-modules are simply called Hermitian, while $-1$-Hermitian $D$-modules are called skew-Hermitian.  

\vskip 5pt

Given a right $\varepsilon$-Hermitian $D$-module $(V,B)$, we can construct a left $\varepsilon$-Hermitian $D$-module $(V^{\ast},B^{\ast})$ in the following way: as a set, $V^{\ast}$ will be the set of symbols $\{v^{\ast}\, |\, v\in V\}$. Then we give $V^{\ast}$ a left $D$-module structure by setting, for all $v$, $w \in V$, $a\in D$, 
\[  \text{$v^{\ast}+w^{\ast}=(v+w)^{\ast}$ and $av^{\ast}=(v\overline{a})^{\ast}$.} \]
 Finally, we set 
\[
 B^{\ast}(v^{\ast},w^{\ast})=\overline{B(w,v)} \qquad \mbox{for all $v$, $w\in V$.}
\]
 In an analogous way, if $V$ is a left $D$-module, we can define a right $D$-module $V^{\ast}$, and $V^{\ast\ast}$ is naturally isomorphic with $V$.  Given $T\in End_{D}(V)$, we can define $T^{\ast}\in End_{D}(V^{\ast})$ by setting $v^{\ast}T^{\ast}:=(Tv)^{\ast}$. With this definition, it is easily seen that $(TS)^{\ast}=S^{\ast}T^{\ast}$, for all $S$, $T\in End_{D}(V)$. Therefore the map $g\mapsto (g^{\ast})^{-1}$ defines an algebraic group isomorphism between $GL(V)$ and $GL(V^{\ast})$. 
 
 \vskip 5pt
 
 Now observe that the form $B$ induces a left $D$-module isomorphism $B^{\flat}:V^{\ast}\longrightarrow V'$ given by $B^{\flat}(v^{\ast})(w)=B(v,w)$ for $v$, $w\in V$. In what follows, we will make implicit use of this map to identify this two spaces. With this identification we can think of $T^{\ast}$ as a map in $End_{D}(V)$ defined by $v^{\ast}(T^{\ast}w):=(v^{\ast}T^{\ast})(w)$, i.e, $T^{\ast}$ is defined by the condition that
\[
 B(v,T^{\ast}w)=B(Tv,w) \qquad \mbox{for all $v$, $w \in V$}.
\]
Observe that this agrees with the usual definition of $T^{\ast}$.
\vskip 5pt

A $D$-submodule $X\subset V$ is said to be \emph{totally isotropic} if $B|_{X\times X}=0$. If $X$ is a totally isotropic submodule, then there exists a totally isotropic submodule $Y\subset V$ such that $B|_{X\oplus Y\times X\oplus Y}$ is nondegenerate.  If we set
\[
 U=(X\oplus Y)^{\perp}:=\{u\in V\, | \, \mbox{$B(u,w)=0$ for all $w\in X\oplus Y$}\}, 
\]
then $V=X\oplus Y \oplus U$, and $B|_{U\times U}$ is non-degenerate. In this case we say that $X$ and $Y$ are totally isotropic, \emph{complementary} submodules. Observe that then $B^{\flat}|_{Y^{\ast}}:Y^{\ast}\longrightarrow X'$ is an isomorphism. As before we will make implicit use of this isomorphism to identify $Y^{\ast}$ with $X'$. 
\vskip 5pt

\subsection{\bf Reductive dual pairs.}
Let $(V,B_{V})$ be a right $\varepsilon_{V}$-Hermitian $D$-module and $(W,B_{W})$ a right $\varepsilon_{W}$-Hermitian $D$-module such that $\varepsilon_{V}\varepsilon_{W}=-1$. On the $\k$-vector space $V\otimes_{D} W^{\ast}$ we can define a symplectic form $B$ by setting
\[
 B(v_{1}\otimes_{D} \lambda_{1}, v_{2}\otimes_{D} \lambda_{2})=\Tr(B_{W}(w_{1},w_{2})B_{V}^{\ast}(\lambda_{2},\lambda_{1})) \qquad \mbox{for all $v_{1}$, $v_{2}\in V$ and $\lambda_{1}$, $\lambda_{2} \in V^{\ast}$.}
\]
Let
\[
 Sp(V\otimes_{D} W^{\ast})=\{g\in GL(V\otimes_{D} W^{\ast},\k) \, | \, \mbox{$B(g v, g w)= B(v, w)$ for all $v$, $w \in V\otimes_{D} W^{\ast}$} \}.
\]
Observe that 
\[  Sp(V\otimes_{D} W^{\ast})=G(V\otimes_{D} W^{\ast},B)=G(V\otimes_{D} W^{\ast}). \]
 Moreover, there is a natural map $G(V) \times G(W) \longrightarrow Sp(V\otimes_{D} W^{\ast})$ given by 
\[
 (g_{1},g_{2})\cdot v\otimes_{D}\lambda = g_{1}v\otimes \lambda g_{2}^{\ast}.
\]
We will use this map to identify $G(V)$ and $G(W)$ with subgroups of $Sp(V\otimes_{D} W^{\ast})$. These two subgroups are mutual commutants of each other, and is an example of a {\em reductive dual pair}.

\vskip 5pt

\subsection{\bf Metaplectic cover.}
The group  $Sp(V\otimes_{D} W^{\ast})$ has an $S^1$- cover  $Mp(V\otimes_{D} W^{\ast})$ which is called a metaplectic group.
It is known that this $S^1$-cover splits over the subgroups $G(V)$ and $G(W)$, except when $V$ is an odd dimensional quadratic space, in which it does not split over $G(W)$.  In this exceptional case, we shall simply redefine $G(W)$ to be the induced double cover, so as to simplify notation.  
We remark also that though the splittings (when they exist) are not necessarily unique, the precise choice of the splittings is of secondary importance in this paper. 

\vskip 5pt

\subsection{\bf Siegel parabolic.}
 Assume in addition that there is a complete polarization
$
 W=E\oplus F,
$
where $E$, $F$, are complementary totally isotropic subspaces of $W$.
We will use the $\varepsilon_{W}$-Hermitian form $B_{W}$ to identify $F^{\ast}$ with $E'$ by setting $f^{\ast}(e)=B_{W}(f,e)$. Observe that this identification induces an identification between $E^{\ast}$ and $F'$ given by
\[
 e^{\ast}(f)=\overline{f^{\ast}(e)}=\overline{B_{W}(f,e)}=\varepsilon_{W}B_{W}(e,f).
\]
In what follows, we will use this identifications between $F^{\ast}$ and $E'$, and between $E^{\ast}$ and $F'$. 

\vskip 5pt

Let 
\[  P=\{p\in G(W)\, | \, p\cdot E =E\} \]
be the Siegel parabolic subgroup of $G(W)$, and let $P=MN$ be its Langlands decomposition.
To give a description of the groups $M$ and $N$, we introduce some more notation.
\vskip 5pt

Let $A\in End_{D}(E)$. We will define $A^{\ast}\in End_{D}(F)$, by setting, for all $e\in E$, $f \in F$,
\begin{equation}
 B_{W}(e,A^{\ast}f) = B_{W}(Ae,f).
\end{equation}
Now given $T\in Hom_{D}(F,E)$, define $T^{\ast}\in Hom_{D}(F,E)$ by setting, for all $f_{1}$, $f_{2}\in F$,
\begin{equation}
B_{W}(f_{1},T^{\ast}f_{2})  =  \varepsilon_{W}B_{W}(Tf_{1},f_{2}).
\end{equation}
Given $\varepsilon=\pm 1$, set 
\[
 Hom_{D}(F,E)_{\varepsilon}=\{T\in Hom_{D}(F,E)\, | \, T^{\ast}=\varepsilon T\}.
\]
It is then clear that $Hom_{D}(F,E)=Hom_{D}(F,E)_{1}\oplus Hom_{D}(F,E)_{-1}$.
 \vskip 5pt

Now we have:
\[
 M=\left.\left\{\left[\begin{array}{cc} A & \\ & (A^{\ast})^{-1} \end{array}\right] \, \right| \, A \in GL(E) \right\} \cong GL(E)
\]
and
\[
 N=\left.\left\{\left[\begin{array}{cc} 1 & X \\ & 1  \end{array}\right] \, \right| \, X^{\ast}=-\varepsilon_{W}X \right\} \cong Hom_{D}(F,E)_{-\varepsilon_{W}}. 
\]

\vskip 5pt

\subsection{\bf Characters of $N$.}
Given $Y\in Hom_{D}(E,F)_{-\varepsilon_{W}}$, define a character
\[
 \chi_{Y}\left(\left[\begin{array}{cc} 1 & X \\ & 1  \end{array}\right]\right)=\chi(\Tr_{F}(YX)).
\]
Here $\Tr_{F}$ is the trace of $YX:F\longrightarrow F$ seen as a map between $\k$ vector spaces. The map $Y\mapsto \chi_{Y}$ defines a group isomorphism between $Hom_{D}(E,F)_{-\varepsilon_{W}}$ and $\hat{N}$. 
\vskip 5pt

Observe that the adjoint action of $M$ on $N$  induces an action of $M$ on $\hat{N}$. Using the isomorphisms of $M\cong GL(E)$ and $\hat{N}\cong Hom_{D}(E,F)_{-\varepsilon_{W}}$, we can describe the action of $M$ on $\hat{N}$ by the formula 
\[  A\cdot  Y = (A^{\ast})^{-1}Y A^{-1} \quad \text{for all $A\in GL(E)$, $Y\in Hom_{D}(E,F)_{-\varepsilon_{W}}$. } \]
Given $Y\in Hom_{D}(E,F)_{-\varepsilon_{W}}$ we can define a $-\epsilon_W$-Hermitian form on $E$, that we will also denote Y, by setting
\[
 Y(e_{1},e_{2})=e_{1}^{\ast}(Ye_{2})=\varepsilon_{W}B_{W}(e_{1},Ye_{2}).
\]
Hence the action of $M$ on $\hat{N}$ is equivalent to the action of $GL(E)$ on sesquilinear, $-\varepsilon_{W}$-Hermitian forms on $E$.

 
\vskip 5pt

 Let $\Omega$ be the set of orbits for the action of $M$ on $\hat{N}$. Given $Y\in Hom_{D}(E,F)_{-\varepsilon_{W}}$, let $\Orb = \Orb_{Y}$ be its orbit under the action of $GL(E)$ and set 
\[
 M_{\chi_{Y}}= \{m\in M \, | \, \mbox{$\chi_{Y}(m^{-1}nm)=\chi_{Y}(n)$ for all $n\in N$}\}.
\]
Using the identification of $M$ with $GL(E)$, and of $\hat{N}$ with $Hom_{D}(E,F)_{-\varepsilon_{W}}$, we see that
\begin{eqnarray*}
 M_{\chi_{Y}} \cong \{A \in GL(E) \, | \, (A^{\ast})^{-1}YA^{-1}=Y\}=\{A \in GL(E) \, | \, Y=A^{\ast}YA\}.
\end{eqnarray*}
\vskip 10pt

\section{\bf Oscillator Representation}

After the preparation of the previous section, we can now consider the theta correspondence associated to the dual pair $G(V) \times G(W)$ and use it to establish certain cases of the  Sakellaridis-Venkatesh conjecture for classical groups.
\vskip 5pt

\subsection{\bf Oscillator representation and theta correspondence.}
    Fix a nontrivial unitary character $\chi$ of $\k$. Associated to this character,  there exists a very special representation of the metaplectic group, called the oscillator representation $\Pi$ of $Mp(V\otimes_{D} W^{\ast})$.
On restricting this representation to $G(V) \times G(W)$,  one obtains an injective map
\[
 \theta:A\subset G(W)^{\wedge} \longrightarrow G(V)^{\wedge}
\]
and a measure $\mu_{\theta}$ on $\widetilde{G}(W)^{\wedge}$, such that
\begin{equation} \label{abstract}
 \Pi|_{G(W)\times G(V)} = \int_{A} \pi\otimes \theta(\pi) \, d\mu_{\theta}(\pi),
\end{equation}
as a $G(W)\times G(V)$-module. 
\vskip 5pt

We may restrict $\Pi$ further to $P \times G(V)$.
By Mackey theory, for a unitary representation $\pi$ of $G(W)$, 
 \begin{equation}
 \pi|_{P}=\bigoplus_{\Orb_{Y}\in \Omega} \Ind_{M_{\chi_{Y}}N}^{P} W_{\chi_{Y}}(\pi), \label{eq:restrictiontoparabolic}
\end{equation}
where $W_{\chi_{Y}}(\pi)$ is an $M_{\chi_{Y}}N$-module such that $n\cdot \lambda=\chi_{Y}(n)\lambda$, for all $n \in N$, $\lambda \in W_{\chi_{Y}}(\pi)$.
Therefore, from  (\ref{abstract}) and
(\ref{eq:restrictiontoparabolic}), we have:
\begin{equation}
\Pi = \bigoplus_{\Orb_{Y}\in \Omega}\int_{A\subset \widehat{G}(W)} \Ind_{M_{\chi_{Y}}N}^{P}W_{\chi_{Y}}(\pi)\otimes \Theta(\pi) \, d\mu_{\theta}(\pi). \label{eq:L2Vabstract}
\end{equation}
\vskip 5pt

\subsection{The Schr\"odinger model}
On the other hand, we may compute the restriction of $\Pi$ to $P \times G(V)$ using an explicit model of $\Pi$. The complete polarization
$ W=E\oplus F$ induces a complete polarization
\[
 V\otimes_{D}W^{\ast}=V\otimes_{D} E^{\ast} \oplus V\otimes_{D} F^{\ast}.
\]
With the identifications introduced above, $V\otimes_{D}F^{\ast}=Hom_{D}(E,V)$, and the oscillator representation $\Pi$ can be realized on the Hilbert space $L^{2}(Hom_{D}(E,V))$; this realization of $\Pi$ is called the Schrodinger model. The action of $P \times G(V)$ in this model can be described as follows.

 \vskip 5pt
 
 Let $B_{V}^{\flat}:V\longrightarrow (V^{\ast})'$ be given by
\[
 (w^{\ast})(B_{V}^{\flat}v)=B_{V}(w,v).
\]
 Then  the action of $P\times G(V)$ on $L^{2}(Hom_{D}(E,V))$ is given by the formulas
\begin{eqnarray}
 \left[\begin{array}{cc} 1 & X \\ & 1 \end{array}\right] \cdot \phi(T) & = & \chi(\Tr_{F}(X T^{\ast}B_{V}^{\flat}T))\phi(T), \label{eq:firstPtimesUaction} \qquad \mbox{for all $X \in Hom_{D}(F,E)_{-\varepsilon_{W}}$}, \\ 
\left[\begin{array}{cc} A &  \\ & (A^{\ast})^{-1} \end{array}\right] \cdot \phi(T) & = & |\mbox{$\det_{F}(A)$}|^{-\dim_{D}(V)/2}\phi(T A),\qquad \mbox{for all $A \in GL(E)$}, \\
g\cdot \phi(T) & = & \phi(g^{-1} T), \qquad \mbox{for all $g \in G(V)$}. \label{eq:lastPtimesUaction}
\end{eqnarray}
Let
$$
\Omega_{V}=\{\Orb_{Y}\, |  \, \mbox{$\Orb_{Y}$ is open in $Hom_{D}(E,F)_{-\varepsilon_{W}}$, and $Y= T^{\ast}B_{V}^{\flat}T$ for some $T\in Hom_{D}(E,V)$}\}.
$$
Given $\Orb_{Y}\in \Omega_{V}$, we will set
\[
\Upsilon_{Y}=\{T\in Hom_{D}(E,V) \, | \, T^{\ast}B_{V}^{\flat}T\in \Orb_{Y}\}.
\]
Then
\[
 \bigcup_{\Orb_{Y}\in \Omega_{V}} \Upsilon_{Y} \subset Hom_{D}(E,V)
\]
is a dense open subset, and its complement in $Hom_{D}(E,V)$ has measure 0. Therefore
\begin{equation}
 L^{2}(Hom_{D}(E,V)) \cong \bigoplus_{\Orb_{Y} \in \Omega_{V}} L^{2}(\Upsilon_{Y}) \label{eq:L2V-omegaVorbits}
\end{equation}
and each of these spaces is clearly $P\times G(V)$-invariant, according to the formulas given in equations (\ref{eq:firstPtimesUaction})--(\ref{eq:lastPtimesUaction}). 
\vskip 5pt

We want to show that the spaces $L^{2}(\Upsilon_{Y})$ are equivalent to some induced representation for $P\times G(V)$. To do this, observe that the ``geometric'' part of the action of $P\times G(V)$ on $L^{2}(\Upsilon_{Y})$ is transitive on $\Upsilon_{Y}$. In other words, under the action of $P\times G(V)$ on $Hom_{D}(E,V)$ given by
\[
 \left(\left[\begin{array}{cc} A & X \\ & (A^{\ast})^{-1} \end{array}\right],g\right)\cdot T= g T A^{-1} \qquad \mbox{for all $\left[\begin{array}{cc} A & X \\ & (A^{\ast})^{-1} \end{array}\right] \in P$, $g\in G(V)$ and $T\in Hom_{D}(E,V)$,}
\]
each of the $\Upsilon_{Y}$'s is a single orbit. Fix $T_{Y}\in \Upsilon_{Y}$   such that $T^{\ast}_{Y}B_{V}^{\flat}T_{Y}=Y$. The stabilizer of $T_{Y}$ in $P\times G(V)$ is the subgroup
\[
 (P\times G(V))_{T_{Y}} =  \left.\left\{\left(\left[\begin{array}{cc} A & X \\ & (A^{\ast})^{-1} \end{array}\right],g\right) \in P\times G(V) \, \right| \, gT_{Y}=T_{Y}A\right\}.
\]
Let $g\in G(V)$ be such that $gT_{Y}=T_{Y}A$ for some $A\in GL(E)$. Then by the definition of $G(V)$
\[
 Y=T_{Y}^{\ast}B_{V}^{\flat}T_{Y}=T_{Y}^{\ast}g^{\ast}B_{V}^{\flat}gT_{Y}=A^{\ast}YA,
\]
that is, $A$ is an element in $M_{\chi_{Y}}$. 
\vskip 5pt

Define an equivalence relation in $Hom_{D}(E,V)$ by setting $T \sim S $ if $T=SA$ for some $A\in M_{\chi_{Y}}$. Given $T\in Hom_{D}(E,V)$ we will denote its equivalence class, under this equivalence relation, by $[T]$. Let 
\[  P_{M_{\chi_{Y}}}(Hom_{D}(E,V))=\{[T]\, | \, T\in Hom_{D}(E,V)\}.\]
 Since $G(V)$ acts by left multiplication on $Hom_{D}(E,V)$, there is natural action of $G(V)$ on the space $P_{M_{\chi_{Y}}}(Hom_{D}(E,V))$. Set
\[
 G(V)_{T_{Y}} = \{g\in G(V)\, | \, gT_{Y}=T_{Y}\}
\quad \text{and} \quad
 G(V)_{[T_{Y}]} = \{g\in G(V)\, | \, g[T_{Y}]=[T_{Y}]\}.
\]
Then $(P\times G(V))_{T_{Y}}\subset M_{\chi_{Y}}\times G(V)_{[T_{Y}]}$, and according to equations (\ref{eq:firstPtimesUaction})-(\ref{eq:lastPtimesUaction}),
\begin{eqnarray}
  L^{2}(\Upsilon_{Y}) & \cong & \Ind_{(P\times G(V))_{T_{Y}}}^{P\times G(V)} \chi_{Y} \\
& \cong & \Ind_{M_{\chi_{Y}}N\times G(V)_{[T_{Y}]}}^{P\times G(V)}\Ind_{(P\times G(V))_{T_{Y}}}^{M_{\chi_{Y}}N\times G(V)_{[T_{Y}]}} \chi_{Y} \label{eq:inductioninstages}
\end{eqnarray}
Now consider the short exact sequence
\[
 1\longrightarrow 1\times G(V)_{T_{Y}}\longrightarrow (P\times G(V))_{T_{Y}} \stackrel{q}{\longrightarrow} M_{\chi_{Y}}N\longrightarrow 1,
\]
where $q$ is the projection into the first component. Observe that the map $q$ induces an isomorphism $G(V)_{T_{Y}}\backslash G(V)_{[T_{Y}]} \cong M_{\chi_{Y}}$. From this exact sequence and equation (\ref{eq:inductioninstages}), we get that
\begin{eqnarray}
 L^{2}(\Upsilon_{Y})  & \cong &  \Ind_{M_{\chi_{Y}}N\times G(V)_{[T_{Y}]}}^{P\times G(V)} L^{2}(G(V)_{T_{Y}}\backslash G(V)_{[T_{Y}]})_{\chi_{Y}} \nonumber\\
& \cong & \Ind_{M_{\chi_{Y}}N}^{P}L^{2}(G(V)_{T_{Y}}\backslash G(V))_{\chi_{Y}}. \label{eq:L2Xdexplicit}
\end{eqnarray}
The action of $M_{\chi_{Y}}N$ on $L^{2}(G(V)_{T_{Y}}\backslash G(V)_{[T_{Y}]})_{\chi_{Y}}$ is given as follows: $N$ acts by the character $\chi_{Y}$, and $M_{\chi_{Y}}$ acts on $L^{2}(G(V)_{T_{Y}}\backslash G(V)_{[T_{Y}]})_{\chi_{Y}}$ on the left using the isomorphism $G(V)_{T_{Y}}\backslash G(V)_{[T_{Y}]} \cong M_{\chi_{Y}}$. Then according to equations (\ref{eq:L2V-omegaVorbits}) and (\ref{eq:L2Xdexplicit})
\begin{equation}
 L^{2}(Hom_{D}(E,V)) \cong \bigoplus_{\Orb_{Y} \in \Omega_{V}} \Ind_{M_{\chi_{Y}}N}^{P}L^{2}(G(V)_{T_{Y}}\backslash G(V))_{\chi_{Y}}. \label{eq:L2Vexplicit}
\end{equation}
But now, from equations (\ref{eq:L2Vabstract}), (\ref{eq:L2Vexplicit}) and the uniqueness of the decomposition of the $N$-spectrum, we obtain:

\begin{proposition}  
As an $M_{\chi_{Y}}N\times G(V)$-module, 
 \begin{equation}
 L^{2}(G(V)_{T_{Y}}\backslash G(V))_{\chi_{Y}} \cong \int_{A \subset \widehat{G}(W)} W_{\chi_{Y}}(\pi)\otimes \Theta(\pi) \, d\mu_{\theta}(\pi), \label{eq:L2(G(Vd)|G(V))decomposition}
\end{equation}
  \end{proposition}
\vskip 5pt

Our goal now is to give a more explicit characterization of the spaces $W_{\chi_{Y}}(\pi)$ and the measure $\mu_{\theta}$ appearing in this formula.
\vskip 5pt

\subsection{Stable range.}
 
Let $(V,B_{V})$ and $(W,B_{W})$ be as before.  Assume now that there is a totally isotropic $D$-submodule $X\subset V$ such that $\dim_{D}(X)=\dim_{D}(W)$; in other words,  the dual pair $(G(V),G(W))$ is in the \emph{stable range}. In this case, the map 
\[  \theta:\widehat{G}(W)\longrightarrow \widehat{G}(V) \]
 can be understood in terms of the results of Jian-Shu Li \cite{J.S. Li:Singular}. The measure $\mu_{\theta}$ appearing in equation (\ref{abstract}) is also known in this case: it is precisely the Plancherel measure of the group $G(W)$. In order to make this paper more self-contained, we will include an alternative calculation of the measure $\mu_{\theta}$ using the so-called \emph{mixed model} of the oscillator representation.

\vskip 5pt

\subsection{\bf Mixed model.}
Let $X$, $Y$ be a totally isotropic, complementary subspaces of $V$ such that $\dim_{D}(X)=\dim_{D}(W)$, and let $U=(X\oplus Y)^{\perp}$. We will use $B_{V}$ to identify $Y$ with $(X^{\ast})'$ by setting
\[
 (x^{\ast})y=B_{V}(x,y), \qquad \mbox{for all $x\in X$, $y\in Y$.}
\]
Given $A\in GL(X)$, we can use the above identification to define an element $A^{\ast}\in GL(Y)$ in the following way: given $x\in X$ and $y\in Y$, we will set $(x^{\ast})(A^{\ast}y) :=  (x^{\ast}A^{\ast})y$, i.e., we will define $A^{\ast} \in GL(Y)$ by requiring that
\[
 B_{V}(x,A^{\ast}y)  =  B_{V}(Ax,y), \qquad \mbox{for all $x\in X$, $y\in Y$.}
\]
Observe that the map $A\mapsto (A^{\ast})^{-1}$ defines an isomorphism between $GL(X)$ and $GL(Y)$. Furthermore if $x\in X$, $y\in Y$ and $A\in GL(X)$, then
\[
 B_{V}(Ax,(A^{\ast})^{-1}y)=B_{V}(x,y).
\]
Therefore, we can define a map $GL(X)\times G(U) \hookrightarrow G(V)$ 
  that identifies $GL(X)\times G(U)$ with the subgroup of $G(V)$ that preserves the direct sum decomposition $V=X\oplus Y \oplus U$. 

\vskip 5pt

Consider the polarization $V\otimes_{D} W^{\ast}=(X\otimes W^{\ast}\oplus U\otimes F^{\ast}) \bigoplus (Y\otimes W^{\ast}\oplus U\otimes E^{\ast})$.
Then as a vector space
\begin{equation}
 L^{2}(X\otimes W^{\ast}\oplus U\otimes F^{\ast})\cong L^{2}(Hom_{D}(W,X)) \otimes L^{2}(Hom_{D}(E,U)). \label{eq:mixedmodel}
\end{equation}
Let $(\omega_{U}, L^{2}(Hom_{D}(E,U)))$ be the Schr\"odinger model of the oscillator representation associated to the metaplectic group $\widetilde{Sp}(U\otimes_{D}W^{\ast})$. We will identify the space appearing on the right hand side of equation (\ref{eq:mixedmodel}) with the space of $L^{2}$ functions from $Hom_{D}(W,X)$ to $L^{2}(Hom_{D}(E,U))$. This is the so called \emph{mixed model} of the oscillator representation.

\vskip 5pt

 The action of $G(W)\times GL(X)\times G(U)$ on this model can be described in the following way: If $T\in Hom_{D}(W,X)$ and $S\in Hom_{D}(E,U)$, then 
\begin{eqnarray}
 g\cdot \phi (T)(S) & = & [\omega_{U}(g)\phi( Tg)](S) \qquad \mbox{$\forall g\in G(W)$}  \label{eq:mixedmodel1}\\
h\cdot \phi (T)(S) & = & \phi(T)(h^{-1}S) \qquad \mbox{$\forall h \in G(U)$}  \label{eq:mixedmodel2} \\
A\cdot \phi (T)(S) & = & |\mbox{$\det_{X}(A)$}|^{\dim{W}/2}\phi(A^{-1}T)(S) \qquad \mbox{$\forall A\in GL(X)$}. \label{eq:mixedmodel3}
\end{eqnarray}
\vskip 5pt

 We now want to describe this space as an induced representation. To do this, observe that the set of invertible elements in $Hom_{F}(W,X)$ forms a single orbit under the natural action of $G(W)\times GL(X)$. Furthermore this orbit is open and dense, and its complement has measure $0$. Fix $T_{0}\in Hom_{F}(W,X)$ invertible, and define a $\varepsilon_{W}$-Hermitian form $B_{T_{0}}$ on $X$, by setting
\[
 B_{T_{0}}(x_{1},x_{2})=B_{W}(T_{0}^{-1}x_{1},T_{0}^{-1}x_{2}).
\]
The group that preserves this form is precisely
\[
 G(X,B_{T_{0}})=\{T_{0}gT_{0}^{-1} \, | \, g\in G(W)\} \subset GL(X).
\]
Let
\[
 (G(W)\times GL(X))_{T_{0}}=\{(g,T_{0}gT_{0}^{-1}) \, | \, g\in G(W)  \} \cong G(W)
\]
be the stabilizer of $T_{0}$ in $G(W)\times GL(X)$. Then, according to equations (\ref{eq:mixedmodel1})--(\ref{eq:mixedmodel3}),
\begin{eqnarray*}
 L^{2}(W\otimes X)\otimes L^{2}(Hom_{D}(E,U)) & \cong & \Ind_{(G(W)\times GL(X))_{T_{0}}}^{G(W)\times GL(X)} L^{2}(Hom_{D}(E,U)) \\
& \cong & \Ind_{G(W) \times G(X,B_{T_{0}})}^{G(W)\times GL(X)}\Ind^{G(W) \times G(X,B_{T_{0}})}_{(G(W)\times GL(X))_{T_{0}}} L^{2}(Hom_{D}(E,U)). 
\end{eqnarray*}
Here $(G(W)\times GL(X))_{T_{0}}$ is acting on $L^{2}(Hom_{D}(E,U))$ by taking projection into the first component, and then using the oscillator representation to define an action of $G(W)$ on $L^{2}(Hom_{D}(E,U))$. But this representation is equivalent to taking projection into the second component and using the Schr\"oridnger model of the oscillator representation of $\widetilde{Sp}(U\otimes X^{\ast})$ (recall that $X$ is equipped with the form $B_{T_0}$)  to define an action of $G(X,B_{T_{0}})$ on $L^{2}(Hom_{D}(T_{0}(E),U))$. Therefore
\begin{eqnarray}
 &L^{2}(W\otimes X)\otimes L^{2}(Hom_{D}(E,U)) \nonumber \\
  \cong & \Ind_{G(W) \times G(X,B_{T_{0}})}^{G(W)\times GL(X)}\Ind^{G(W) \times G(X,B_{T_{0}})}_{(G(W)\times GL(X))_{T_{0}}}  L^{2}(Hom_{D}(T_{0}(E),U))\nonumber \\
 \cong & \Ind_{G(W) \times G(X,B_{T_{0}})}^{G(W)\times GL(X)} \int_{\widehat{G}(W)} \pi^{\ast} \otimes (\pi^{T_{0}}  \otimes L^{2}(Hom_{D}(T_{0}(E),U)))\, d\mu_{G(W)}(\pi) \nonumber \\
 \cong & \int_{\widehat{G}(W)}  \pi^{\ast} \otimes \Ind_{G(X,B_{T_{0}})}^{GL(X)} \pi^{T_{0}}  \otimes L^{2}(Hom_{D}(T_{0}(E),U)) \, d\mu_{G(W)}(\pi). \label{eq:plancherelmeasure}
\end{eqnarray}
Here $\pi^{\ast}$ is the contragredient representation of $\pi$,  $\pi^{T_{0}}$ is the representation of $G(X,B_{T_{0}})$ given by $\pi^{T_{0}}(g)=\pi(T_{0}^{-1}gT_{0})$, for all $g\in G(X,B_{T_{0}})$, and $\mu_{G(W)}$ is the Plancherel measure of $G(W)$. Note that the multiplicity space of $\pi^{\ast}$ in (\ref{eq:plancherelmeasure}) is nonzero for each $\pi$ in the support of 
$\mu_{G(W)}$, i.e. as a representation of $G(W)$, $\Pi$ is weakly equivalent to the regular representation $L^2(G(W))$.
\vskip 5pt

Comparing (\ref{abstract}) with (\ref{eq:plancherelmeasure}), we obtain:
\vskip 5pt

\begin{proposition}
If $(G(W), G(V))$ is in the stable range, with $G(W)$ the smaller group, then in equations (\ref{abstract}) and (\ref{eq:L2(G(Vd)|G(V))decomposition}),
 \[ A = \widehat{G}(W) \quad \text{and} \quad \mu_{\theta} = \mu_{G(W)}. \]
\end{proposition}
\vskip 10pt

\subsection{\bf The Bessel-Plancherel theorem.}
Finally, we want to identify the multiplicity space $W_{\chi_Y}(\pi)$ in 
(\ref{eq:L2(G(Vd)|G(V))decomposition}). Note that this is purely an issue about representations of $G(W)$; a priori, it has nothing to do with theta correspondence. What we know is summarized in the following theorem.
\vskip 5pt

\begin{theorem}[Bessel-Plancherel theorem]  \label{T:bessel}
Let $(W,B_{W})$ be an $\varepsilon_{W}$-Hermitian $D$-module, and assume that $W$ has a complete polarization $W=E\oplus F$, where $E$, $F$ are totally isotropic complementary subspaces. Let $P=\{p\in G(W)\, | \, p \cdot E=E\}$ be a Siegel parabolic subgroup of $G$, and let $P=MN$ be its Langlands decomposition. Given $\chi\in \hat{N}$, let $\Orb_{\chi}$ be its orbit under the action of $M$, and let $M_{\chi}$ be the stabilizer of $\chi$ in $M$. Then
\begin{enumerate}
 \item For $\mu_{G(W)}$-almost all tempered representation $\pi$ of $G(W)$,
\[
 \pi|_{P}  \cong \bigoplus_{\Orb_{\chi}\in \Omega_{W}}\Ind_{M_{\chi}N}^{P}W_{\chi}(\pi).
\]
Here $\mu_{G(W)}$ is the Plancherel measure of $G(W)$,  $\Omega_{W}=\{\Orb_{\chi}\in \Omega\, | \, \mbox{$\Orb_{\chi}$ is open in $\hat{N}$}\}$, and $W_{\chi}(\pi)$ is some $M_{\chi}N$-module such that the action of $N$ is given by the character $\chi$.
\vskip 5pt

\item If $\Orb_{\chi}\in \Omega_{W}$, then there is an isomorphism of $M_{\chi} \times G(W)$-modules:
\begin{equation}
 L^{2}(N\backslash G(W); \chi) \cong \int_{\widehat{G}(W)} W_{\chi}(\pi) \otimes \pi \, d\mu_{G(W)}(\pi). \label{eq:BesselPlancherel}
\end{equation}
where $W_{\chi}(\pi)$ is the same space appearing in (1).
\vskip 5pt

\item   If  $\dim_{D}(W)=2$, then for $\Orb_{\chi}\in \Omega_{W}$,  $\dim W_{\chi}(\pi) < \infty$ and
\[
 W_{\chi}(\pi)\cong Wh_{\chi}(\pi)=\{\lambda:\pi^{\infty}\longrightarrow \C\, | \, \mbox{$\lambda(\pi(n)v)=\chi(n)\lambda(v)$ for all $n\in N$}\}
\]
as an $M_{\chi}N$-module. Here $\pi^{\infty}$ stands for the set of $C^{\infty}$ vectors of $\pi$.
  \vskip 5pt

  \item If $\k$ is Archimedean, and $M_{\chi}$ is compact, then
  \[  W_{\chi}(\pi) \subset Wh_{\chi}(\pi) \]
  as a dense subspace, and for any irreducible representation $\tau$ of $M_{\chi}$, one has an equality of $\tau$-isotypic parts:
  \[  W_{\chi}(\pi)[\tau] =  Wh_{\chi}(\pi)[\tau]. \]
Moreover, this space  is finite dimensional.
 \end{enumerate}
\end{theorem}

\begin{proof}
 Part 2 follows from an argument analogous to the proof of the Whittaker-Plancherel measure given by Sakellaridis-Venkatesh \cite[\S 6.3]{Sakellaridis-Venkatesh}.  For the proof of part 1 observe that, by Harish-Chandra Plancherel theorem 
\[
 L^{2}(G(W))|_{P\times G(W)}  = \int_{\widehat{G}(W)} \pi^{\ast}|_{P}\otimes \pi \, d\mu_{G(W)}(\pi).
\]
On the other hand
\begin{align}
 L^{2}(G(W))|_{P\times G(W)} 
 &=  \bigoplus_{\Orb_{Y}\in \Omega_{W}} \Ind_{M_{\chi_{Y}}N}^{P} L^{2}(N\backslash G(W); \chi) \notag \\
&=  \bigoplus_{\Orb_{Y}\in \Omega_{W}} \Ind_{M_{\chi_{Y}}N}^{P} \int_{\widehat{G}(W)} W_{\chi_{Y}}(\pi) \otimes \pi \, d\mu_{G(W)}(\pi) \notag \\
&=  \int_{\widehat{G}(W)} \left[\bigoplus_{\Orb_{Y}\in \Omega_{W}} \Ind_{M_{\chi_{Y}}N}^{P}  W_{\chi_{Y}}(\pi)\right] \otimes \pi \, d\mu_{G(W)}(\pi). \notag
\end{align}
Therefore
\[
 \pi^{\ast}|_{P}  \cong \bigoplus_{\Orb_{Y}\in \Omega_{W}}\Ind_{M_{\chi}N}^{P}W_{\chi}(\pi)
\]
for $\mu_{G(W)}$-almost all $\pi$. In the Archimedean case, this result has also been proved in the thesis of the second named author without the $\mu_{G(W)}$-almost all restriction, yielding an alternative proof of part 2 for the Archimedean case.
\vskip 5pt

Part 3 was proved by Wallach in the Archimedean case \cite{w:vol2}, and independently by Delorme, Sakellaridis-Venkatesh and U-Liang Tang in the $p$-adic case \cite{Delorme,  Sakellaridis-Venkatesh, U-Liang}.
\vskip 5pt

Finally, Part 4 was shown by Wallach and the second named author in \cite{raul-nolan}.
\end{proof}

\vskip 5pt

\subsection{\bf Spectral decomposition of generalized Stiefel manifolds}
We may now assemble all the previous results together.  For $\Orb_{Y}\in \Omega_{V}$, 
the space $G(V)_{T_{Y}}\backslash G(V)$ is known as a \emph{generalized Stiefel manifold}. 
From equations (\ref{eq:L2(G(Vd)|G(V))decomposition}) and (\ref{eq:plancherelmeasure}), we deduce:



\begin{theorem}
Suppose that  $G(V)_{T_{Y}}\backslash G(V)$ is a generalized Stiefel manifold.  If, in the notation of equation (\ref{eq:BesselPlancherel})
\[
L^{2}(N\backslash G(W); \chi_{Y}) \cong \int_{\widehat{G}(W)} W_{\chi_{Y}}(\pi) \otimes \pi \, d\mu_{G(W)}(\pi).
\]
then
\[
 L^{2}(G(V)_{T_{Y}}\backslash G(V)) \cong \int_{\widehat{G}(W)} W_{\chi_{Y}}(\pi)\otimes \Theta(\pi) \, d\mu_{G(W)}(\pi).
\]
 \end{theorem}
 
In a certain sense, the last pair of equations says that the Plancherel measure of the generalized Stiefel manifold $G(V)_{T_{Y}} \backslash G(V)$ is the pushforward of the Bessel-Plancherel measure of $G(W)$ under the $\theta$-correspondence. 
\vskip 10pt

\subsection{\bf The Sakellaridis-Venkatesh conjecture.}
Using the previous theorem, we can obtain certain examples of the Sakellaridis-Venkatesh conjecture:
\vskip 5pt

\begin{itemize}
\item Taking $D = \k$, $\k \times \k$ or $M_2(\k)$ to be a split $\k$-algebra and 
$W$ to be skew-Hermitian with $\dim_D W = 2$, we obtain the spectral decomposition of 
 \[  H \backslash G = O_{n-1} \backslash O_n,\quad   GL_{n-1} \backslash GL_n, \quad Sp_{2n-2} \backslash Sp_{2n}, \]
in terms of the Bessel-Plancherel (essentially the Whittaker-Plancherel) decomposition for  $G_X = GL_2$, $SL_2$, $\tilde{SL}_2$ or $SO_4$. This establishes the cases listed in Table 1 in Theorem 1.
\vskip 5pt

\item Taking $D$ to be a quadratic field extension of $\k$  or the quaternion division $\k$-algebra, and $W$ to be skew-Hermitian, we obtain the spectral decomposition of
\[  H \backslash G = U_{n-1} \backslash U_n, \quad Sp_{n-1}(D) \backslash Sp_{n}(D) \]
in terms of the Bessel-Plancherel decomposition of $U_2$ and $O_2(D)$.
 This gives non-split version of the examples above.
\end{itemize}

\vskip 5pt

In addition, the multiplicity space $W_{\chi}(\pi) = Wh_{\chi}(\pi)$ should be describable in terms of the space of $H$-invariant (continuous) functionals on  $\Theta(\pi)^{\infty}$. 
Indeed, by the smooth analog of our computation with the Schrodinger model in \S 3.2, one can show that there is a natural isomorphism of $M_{\chi}$-modules:
\[  Wh_{\chi}(\pi) \cong  Hom_H( \Theta^{\infty}(\pi^{\infty}),  \C). \]
Here $\Theta^{\infty}(\pi^{\infty})$ refers to the (big) smooth theta lift of the smooth representation $\pi^{\infty}$, i.e. the representation $\pi^{\infty} \boxtimes \Theta^{\infty}(\pi^{\infty})$ is the maximal $\pi^{\infty}$-isotypic quotient of the smooth model $\Pi^{\infty} = \mathcal{S}(Hom_D(E,V))$
of the oscillator representation $\Pi$.   One can  show using the machinery developed in Bernstein's paper \cite{Be} that for $\mu_{\theta}$-almost all $\pi$, one has the compatibility of $L^2$-theta lifts (considered in this paper) with the smooth  theta lifts:
\[  \Theta(\pi)^{\infty} \cong \Theta^{\infty}(\pi^{\infty}). \]
With this compatibility, one will obtain
\[  W_{\chi}(\pi) \cong  Hom_H(\Theta(\pi)^{\infty}, \C). \]







\subsection{\bf Unstable range.} 
Though we have assumed that $(G(W), G(V))$ is in the stable range from \S 3.3, it is possible to say something when one is not in the stable range as well. Namely, in  \S 3.4, one would take $X$ to be a maximal isotropic space in $V$ (so  $\dim X < \dim W$ here), and consider the mixed model defined on $L^2(Hom_D(W,X)) \otimes L^2(Hom_D(E,U))$.  As an illustration, we note  the result for the case when $W$ is a symplectic space of dimension $2$ and $V$ is a split quadratic space of dimension $3$, so that
\[  G(W) \times G(V) \cong \tilde{SL}_2 \times SO_3 \cong \tilde{SL}_2 \times PGL_2. \]
For a nonzero $Y \in \widehat{N}$, the subgroup $G(V)_{T_Y}$ of $G(V)$ is simply a maximal torus $A_Y$ of $PGL_2$. 
\vskip 5pt

\begin{proposition} \label{P:rank1}
We have
\[  L^2(G(V)_{T_Y} \backslash G(V))  = L^2(A_Y \backslash PGL_2) \cong \int_{\widehat{G(W)}}
 (W_{\chi}(\sigma) \otimes W_{\chi_Y}(\sigma))  \otimes \Theta_{\chi}(\pi) \, d \mu_{G(W)}(\pi). \]
\end{proposition}
\vskip 5pt

We record the following corollary which is needed in the second half of this paper:
\vskip 5pt

\begin{corollary} \label{C:rank1}
 The unitary representation $L^2(\mathfrak{sl}_2)$ associated to
the adjoint action of $PGL_2$ on its Lie algebra $\mathfrak{sl}_2$ is
  weakly equivalent to the regular representation $L^2(PGL_2)$.
  \end{corollary}

\begin{proof}
Since the union of strongly regular semisimple classes are open dense in $\mathfrak{sl}_2$, we see that $L^2(\mathfrak{sl}_2)$ is weakly equivalent to $\bigoplus_A  L^2(A \backslash PGL_2)$, where the sum runs over conjugacy classes of maximal tori $A$ in $PGL_2$. 
Applying Proposition \ref{P:rank1}, one deduces that
 \[ \bigoplus_A  L^2(A \backslash PGL_2) \cong \int_{\widehat{G(W)}}  M_{\chi}(\pi) \otimes \Theta_{\chi}(\pi) \, d\mu_{G(W)}(\pi) \]
 with
\[  M_{\chi}(\pi) = W_{\chi}(\pi) \otimes \left( \bigoplus_A W_{\chi_A}(\pi) \right). \] 
One can show that the theta correspondence with respect to $\chi$ induces a bijection
\[  \Theta_{\chi}:  \{ \pi \in \widehat{G(W)}: W_{\chi}(\pi) \ne 0 \} \longleftrightarrow
\widehat{G(V)}. \]
Moreover, one can write down this bijection explicitly (in terms of the usual coordinates on the unitary duals of $\widetilde{SL}_2$ and $PGL_2$). From this description, one sees that
\[  (\Theta_{\chi})_*(\mu_{G(W)})  = \mu_{G(V)}. \]
This shows that
\[ \int_{\widehat{G(W)}}  M_{\chi}(\pi) \otimes \Theta_{\chi}(\pi) \, d\mu_{G(W)}(\pi) \cong
\int_{\widehat{G(V)}} M_{\chi}(\Theta_{\chi}^{-1}(\sigma)) \otimes \sigma \, d\mu_{G(V)}(\sigma), \]
with $M_{\chi}(\Theta_{\chi}^{-1}(\sigma)) \ne 0$. This proves the corollary.
\end{proof} 
\vskip 15pt

\section{\bf Exceptional Structures and Groups}
The argument of the previous section can be adapted to various dual pairs in exceptional groups, thus giving rise to more exotic examples of the Sakellaridis-Venkatesh conjecture. In particular,   we shall show that the spectral decomposition of $L^2(X) = L^2(H\backslash G)$ can obtained from that of $L^2(G_X)$, with $X$ and $G_X$ given in the following table.
\vskip 15pt
\begin{table}[ht]
\label{table2}
\begin{center}
\begin{tabular}{|c|c|c|c|c|c|c|c|c|} 
\hline 
 
$X$ &$SO_3 \backslash SL_3$ & $SL_3\backslash G_2$ &  $(J, \psi) \backslash G_2$ 
& $Sp_6 \backslash SL_6$ & $G_2 \backslash Spin_7$ & $G_2 \backslash Spin_8$ & $Spin_9 \backslash F_4$ & $F_4 \backslash E_6$ \\
 \hline 
 $G_X$ &  $\widetilde{SL}_3$ & $SL_2$ &   $PGL_3$ &   $SL_3$ &  $ SL_2$& $SL_2^3/\Delta \mu_2$ & $PGL_2$ & $SL_3$  \\
 \hline
\end{tabular}
\end{center}
\caption{}
\end{table}

\noindent The unexplained notation will be explained in due course.  
Comparing with the tables in \cite[\S 15 and \S 16]{Sakellaridis-Venkatesh}, we see that these exceptional examples,  together with the classical examples treated earlier,  verify the conjecture of Sakallaridis-Venkatesh for almost all  the rank $1$ spherical varieties (with certain desirable properties), and also some rank $2$ or rank $3$ ones.
\vskip 5pt

Though the proof will be similar in spirit to that of the previous section, we shall need to deal with the geometry of various exceptional groups, and this is ultimately based on the geometry of the (split) octonion algebra $\mathbb{O}$ and the exceptional Jordan algebra $J(\mathbb{O})$. Thus we need to recall some basic properties of $\mathbb{O}$ and its automorphism group. A good reference for the material in this section is the book \cite{KMRT}. One may also consult \cite{MS} and \cite{We}.
\vskip 10pt

\subsection{\bf Octonions and $G_2$.}
Let $\k$ be a local field of characteristic zero and let $\mathbb{O}$ denote the (8-dimensional) split octonion algebra over $\k$. The octonion algebra $\mathbb{O}$ is non-commutative and non-associative. Like the quaternion algebra, it is endowed with a conjugation $x \mapsto \bar{x}$ with an associated trace map $Tr(x) = x + \bar{x}$ and an associated norm map $N(x) = x \cdot \bar{x}$. 
It is a composition algebra, in the sense that $N(x \cdot y) = N(x) \cdot N(y)$. 
\vskip 5pt

A useful model for $\mathbb{O}$ is the so-called Zorn's model, which consists of $2 \times 2$-matrices 
\[  \left( \begin{array}{cc} 
a & v \\
v' & b \end{array} \right), \quad \text{with $a,b \in \k$, $v \in V \cong k^3$ and $v' \in V'$,}  \]
with $V$ a 3-dimensional $\k$-vector space with dual $V'$. Note that there are natural isomorphisms
\[  \wedge^2 V \cong V' \quad \text{and} \quad \wedge^2 V' \cong V,\]
and let $\langle -, -\rangle$ denote the natural pairing on $V' \times V$. 
The multiplication on $\mathbb{O}$ is then defined by
\[  
 \left( \begin{array}{cc} 
a & v \\
v' & b \end{array} \right)   \cdot  \left( \begin{array}{cc} 
c & w \\
w' & d \end{array} \right)
=  
 \left( \begin{array}{cc} 
ac + \langle w',v \rangle   & aw + dv + v' \wedge w'  \\
cv' + bw' + v \wedge w  & bd + \langle v', w \rangle  \end{array} \right)
\]
The conjugation map is
\[   \left( \begin{array}{cc} 
a & v \\
v' & b \end{array} \right) \mapsto  \left( \begin{array}{cc} 
b & -v \\
-v' & a \end{array} \right) \]
so that
\[  Tr  \left( \begin{array}{cc} 
a & v \\
v' & b \end{array} \right) = a+b \quad \text{and} \quad N  \left( \begin{array}{cc} 
a & v \\
v' & b \end{array} \right) = ab - \langle v',v \rangle. \]
Any non-central element $x \in \mathbb{O}$ satisfies the quadratic polynomial 
$x^2 - Tr(x) \cdot x + N(x) = 0$.  Thus, a non-central element $x \in \mathbb{O}$ generates a quadratic $\k$-subalgebra  described by this quadratic polynomial. If this quadratic polynomial is separable, $x$ is said to have rank $2$. Otherwise, $x$ is said to have rank $1$.  

\vskip 5pt
The automorphism group of the algebra $\mathbb{O}$ is the split exceptional group of type $G_2$. 
 The group $G_2$  contains the subgroup $SL(V) \cong SL_3$ which fixes the diagonal elements in Zorn's model, and acts on $V$ and $V'$ naturally. Clearly, $G_2$ fixes the identity element $1 \in \mathbb{O}$, so that it acts on the subspace $\mathbb{O}_0$ of trace zero elements. The following proposition summarizes various properties of the action of $G_2$ on $\mathbb{O}_0$.
\vskip 5pt

\begin{proposition}
(i) Fix $a \in \k^{\times}$, and let $\Omega_a$ denote the subset of $x \in \mathbb{O}_0$ with $N(x) = a$, then $\Omega_a$ is nonempty and $G_2$ acts transitively on $\Omega_a$ with stabilizer isomorphic to $SU_3(E_a)$, where $E_a = \k[x]/(x^2-a)$. 
\vskip 5pt

\noindent (ii) The automorphism group $G_2$ acts transitively on the set $\Omega_0$ of trace zero, rank 1 elements. For $x \in \Omega_0$, the stabilizer of the line $\k \cdot x$ is a maximal parabolic subgroup 
$Q = L \cdot U$ with Levi factor $L \cong GL_2$ and unipotent radical $U$ a 3-step unipotent group. 

\end{proposition}

Now we note:
\begin{itemize}
\item  When $a \in (\k^{\times})^{2}$ in (i), the stabilizer of an element in $\Omega_a$ is
isomorphic to $SL_3$; this explains the 2nd entry in Table 3. 
 
\item In (ii),  the 3-step filtration of $U$ is given by
\[  U \supset [U,U] \supset Z(U) \supset \{1 \} \]
where $[U,U]$ is the commutator subgroup and $Z(U)$ is the center of $U$. Moreover, 
\[ \dim Z(U) =2 \quad \text{and} \quad  \dim [U,U] = 3,\]
so that $[U,U]/ Z(U) \cong \k$. If $\psi$ is a non-trivial character of $\k$, then $\psi$ gives rise to a nontrivial character of $[U,U]$ which is fixed by the subgroup $[L,L] \cong SL_2$. Setting $J = [L,L] \cdot [U,U]$, we may extend $\psi$ to a character of $J$ trivially across $[L,L]$. This explains the 3rd entry of Table 3.
\end{itemize}
\vskip 5pt

Though the octonionic multiplication is neither commutative or associative, the trace form satisfies:
\[  Tr((x \cdot y) \cdot z) = Tr(x \cdot (y \cdot z)),  \]
(so there is no ambiguity in denoting  this element of $\k$ by $Tr( x \cdot y \cdot z)$)  and $G_2$ is precisely the subgroup of $SO(\mathbb{O}, N)$ satisfying
\[  Tr((g x) \cdot (gy) \cdot (gy)) = Tr(x \cdot y \cdot z) \quad \text{  for all $x,y,z \in \mathbb{O}$}. \]
\vskip 10pt

\subsection{\bf Exceptional Jordan algebra and $F_4$.} 

Let $J = J(\mathbb{O})$ denote the 27-dimensional vector space consisting of all $3 \times 3$ Hermitian matrices with entries in $\mathbb{O}$. Then a typical element in $J$ has the form
\[ \alpha = \left( \begin{array}{ccc}
a & z &\bar{y} \\
\bar{z} & b & x \\
y & \bar{x} & c \end{array} \right), \quad \text{with $a,b,c \in \k$ and $x,y,z \in \mathbb{O}$.} \]
The set $J$ is endowed with a multiplication
\[  \alpha \circ \beta = \frac{1}{2}\cdot  (\alpha \beta+ \beta \alpha)  \]
where the multiplication on the RHS refers to usual matrix multiplication. With this multiplication, $J$ is the exceptional Jordan algebra. 
\vskip 5pt

The algebra $J$ carries a natural cubic form $d = \det$ given by the determinant map on $J$, 
and a natural linear form $tr$ given by the trace map. 
Moreover, every element in $J$ satisfies a cubic polynomial, by the analog of the Cayley-Hamilton theorem. An element $\alpha \in J$ is said to be of rank $n$ if its minimal polynomial has degree $n$, so that $ 0 \leq n\leq 3$. For example, $\alpha \in J$ has  rank $1$  if and only if its entries satisfy
\[   N(x) = bc, \, N(y) =ca, \, N(z) = ab,  \,  xy = c \bar{z}, \,  yz = a \bar{x}, \, zx = b \bar{y}. \]

 \vskip 5pt

More generally, the above discussion holds if one uses any composition $\k$-algebra in place of $\mathbb{O}$. Thus, if $B = \k$, a quadratic algebra $K$, a quaternion algebra $D$ or the octonion algebra $\mathbb{O}$, one has the Jordan algebra $J(B)$. One may consider 
the group $\Aut(J(B), \det)$ of invertible linear maps on $J(B)$ which fixes the cubic form $\det$, 
and its subgroup  $\Aut(J, \det, e)$  which fixes an element $e$ with $\det(e) \ne 0$.
 For the various $B$'s, these groups are listed in the following table.
\vskip 5pt
\begin{table}[ht]
\begin{center}
\begin{tabular}{|c|c|c|c|c|}
\hline
$B$ & $\k$ & $K$ & $D$ & $\mathbb{O}$ \\
\hline
$\Aut(J(B), \det)$ & $SL_3$ &  $SL_3 (K)/\Delta \mu_3$ & $SL_3(D)/\mu_2 = SL_6/ \mu_2$ & $E_6$ \\
\hline
$\Aut(J(B), \det, e)$ & $SO_3$ & $SL_3$ & $PGSp_6$ & $F_4$ \\
\hline
\end{tabular}
\caption{}
\end{center}
\end{table}
\vskip 5pt

\begin{proposition} \label{P:jordan}
(i) For any $a \in \k^{\times}$, the group $\Aut(J(B), \det)$ acts transitively on the set of $e \in J$ with $\det(e) = a$, with stabilizer group $\Aut(J(B), \det, e)$ described in the above table. If $e$ is the unit element of $J(B)$, then $\Aut(J(B), \det, e)$ is the automorphism group of the Jordan algebra $J(B)$.
\vskip 5pt

\noindent (ii) The group $F_4 = \Aut(J(\mathbb{O}))$ acts transitively on the set of rank $1$ elements in $J(\mathbb{O})$ of trace $a \ne 0$. The stabilizer of a point is isomorphic to the group $Spin_9$ of type $B_4$. 
\end{proposition}
In particular, the proposition explains the 1st, 4th, 7th and 8th  entry of Table 3.
\vskip 10pt

\subsection{\bf Triality and $Spin_8$}
An element $\alpha \in J = J(\mathbb{O})$ of rank $3$  generates a commutative separable cubic subalgebra $\k(\alpha) \subset J$. For any such cubic $F$-algebra $E$, one may consider the set $\Omega_E$ of algebra embeddings $E \hookrightarrow J$. Then one has:
\vskip 5pt

\begin{proposition}
(i) The set $\Omega_E$ is non-empty and the group $F_4$ acts transitively on $\Omega_E$.
\vskip 5pt

(ii) The stabilizer of a point in $\Omega_E$ is isomorphic to the quasi-split simply-connected group $Spin_8^E$ of absolute type $D_4$. 
\vskip 5pt

(iii) Fix an embedding $j: E \hookrightarrow J$ and let $E^{\perp}$ denote the orthogonal complement of the image of $E$ with respect to the symmetric bilinear form $(\alpha, \beta) = tr(\alpha \circ \beta)$. The action of the stabilizer $Spin_8^E$ of $j$ on $E^{\perp}$ is the 
24-dimensional Spin representation, which on extending scalars to $\overline{\k}$, is the direct sum of the three 8-dimensional irreducible representations of $Spin_8(\overline{\k})$ whose highest weights correspond to the 3 satellite vertices in the Dynkin diagram of type $D_4$. 
\end{proposition}

\vskip 5pt

As an example, suppose that $E = \k \times \k \times \k$, and we fix the natural embedding $E \hookrightarrow J$ whose image is the subspace of diagonal elements in $J$.
 Then $E^{\perp}$ is naturally $\mathbb{O} \oplus \mathbb{O} \oplus \mathbb{O}$, and the split group $Spin_8$ acts on this, preserving each copy of $\mathbb{O}$. This gives  an injective homomorphism
 \[ \rho:  Spin_8  \longrightarrow SO(\mathbb{O}, N) \times 
SO(\mathbb{O}, N) \times SO(\mathbb{O}, N) \]
whose image is given by
\[  Spin_8 \cong \{  g = (g_1, g_2, g_3):  Tr((g_1x) \cdot( g_2 y) \cdot (g_3z)) = Tr(x \cdot y \cdot z) \quad \text{for all $x,y,z \in \mathbb{O}$} \}. \]
From this description, one sees that there is an action of $\Z/3\Z$ on $Spin_8$ given by the cyclic permutation of the components of $g$, and the subgroup fixed by this action is precisely 
\[  G_2 = Spin_8^{\Z/3\Z}. \]
This explains the 6th entry of Table 3.  
\vskip 5pt

More generally , the stabilizer of a triple $(x,y,z) \in \mathbb{O}^3$ with $(x \cdot y) \cdot z \in \k^{\times}$ is a subgroup of $Spin_8$ isomorphic to $G_2$ (see \cite{We}). 
For example,  the stabilizer in $Spin_8$ of the vector $(1,0,0) \in \mathbb{O}^3$ is isomorphic to the group $Spin_7$ which acts naturally on $\mathbb{O}_0 \oplus \mathbb{O} \oplus \mathbb{O}$. The action of $Spin_7$ on $\mathbb{O}_0$ is via  the standard representation of $SO_7$, whereas its action on the other two copies of $\mathbb{O}$ is via the Spin representation. From the discussion above, we see that the stabilizer in $Spin_7$  of $(x, \bar{x}) \in \mathbb{O}^2$, with $N(x) \ne 0$, is isomorphic to the group $G_2$. In particular, this explains the 5th entry of Table 3. 

\vskip 5pt
By the above discussion, it is not difficult to show:
\vskip 5pt

\begin{proposition} \label{P:D4}
The group $Spin_8$ acts transitively on the set of rank $1$ elements in $J(\mathbb{O})$ with diagonal part $(a,b,c) \in \k^{\times} \times \k^{\times} \times \k^{\times}$. Moreover, the stabilizer of a point is isomorphic to $G_2$.
\end{proposition}
\vskip 5pt

\subsection{\bf $SL_3 \backslash G_2$ and $G_2 \backslash Spin_7$.} 
From the discussion above, we see that there are isomorphisms of homogeneous varieties
\[  SL_3 \backslash G_2 \cong SO_6 \backslash SO_7 \quad \text{and} \quad 
  G_2 \backslash Spin_7 \cong Spin_7 \backslash Spin_8 \cong SO_7 \backslash SO_8.\]
Since we have already determined the spectral decomposition of $L^2(SO_6 \backslash SO_7)$ and $L^2(SO_7 \backslash SO_8)$ in terms of the spectral decomposition of $L^2(\tilde{SL}_2)$ and $L^2(SL_2)$ respectively, we obtain the desired description for $SL_3\backslash G_2$ and $G_2 \backslash Spin_7$. Thus the rest of the paper is devoted to the remaining cases in Table 3.

\vskip 10pt
\section{\bf Exceptional Dual Pairs}
In this section, we  introduce some exceptional dual pairs  contained in the  adjoint groups of type $F_4$, $E_6$, $E_7$ and $E_8$. We begin with a uniform construction of the exceptional Lie algebras of the various exceptional groups introduced above. This construction can be found in \cite{R} and will be useful for exhibiting various reductive dual pairs.  The reader may consult \cite{MS}, \cite{R},  \cite{S} and \cite{We} for the material of this section.

\vskip 5pt
\subsection{\bf Exceptional Lie algebras.}
 Consider the chain of Jordan algebras
\[  \k \subset  E  \subset J(\k) \subset J(K) \subset  J(D) \subset J(\mathbb{O}) \]
where $E$ is a cubic $\k$-algebra, $K$ a quadratic $\k$-algebra and $D$ a quaternion $\k$-algebra. Denoting such an algebra by $\mathcal{R}$, the determinant map $\det$  of $J(\mathbb{O})$ restricts to give a cubic form on $\mathcal{R}$. 
Now set
\begin{equation} \label{E:lie}
 \mathfrak{s}_{\mathcal{R}} =  \mathfrak{sl}_3 \oplus \mathfrak{m}_{\mathcal{R}} \oplus( \k^3 \otimes  \mathcal{R}) \oplus ( \k^3 \otimes  \mathcal{R})', \end{equation} 
with 
\[  \mathfrak{m}_{\mathcal{R}} = {\rm Lie}(Aut(\mathcal{R}, \det)). \]
One can define a Lie algebra structure on $\mathfrak{s}_{\mathcal{R}}$ \cite{R} whose type is given by the following table.
\vskip 5pt

\begin{center}
\begin{tabular}{|c|c|c|c|c|c|c|}
 \hline  
 $\mathcal{R}$ & $\k$ & $E$ & $J(\k)$ & $J(K)$ & $J(D)$ & $J(\mathbb{O})$ \\
 \hline 
 $\mathfrak{m}_{\mathcal{R}}$ & $0$ & $E_0$ & $\mathfrak{sl}_3$ & $\mathfrak{sl}_3(K)$ & $\mathfrak{sl}_6$ & $\mathfrak{e}_6$ \\
 \hline 
 $\mathfrak{s}_{\mathcal{R}}$ & $\mathfrak{g}_2$ & $\mathfrak{d}_4$ & $\mathfrak{f}_4$ & $\mathfrak{e}_6$ & $\mathfrak{e}_7$ & $\mathfrak{e}_8$ \\
 \hline
 \end{tabular}
 \end{center} 
 \vskip 5pt
 
 We denote the corresponding adjoint group with Lie algebra $\mathfrak{s}_{\mathcal{R}}$ by $S_{\mathcal{R}}$, or simply by $S$ if $\mathcal{R}$ is fixed and understood.  
  \vskip 5pt

Let $\{e_1, e_2, e_3 \}$ be the standard basis of $\k^3$ with dual basis $\{ e_i' \}$.
The subalgebra of $\mathfrak{sl}_3$ stabilizing the lines $\k e_i$ is the diagonal torus $\mathfrak{t}$.  The nonzero wieghts under the adjoint action of $\mathfrak{t}$ on $\mathfrak{s}_{\mathcal{R}}$ form a root system of type $G_2$.  The long root spaces are of dimension $1$ and are precisely the root spaces of $\mathfrak{sl}_3$, i.e. the spaces spanned by $e_i' \otimes e_j$.  We shall label these long roots by $\beta$, $\beta_0$ and $\beta_0 - \beta$, with corresponding 1-parameter subgroups
\[  u_{\beta}(x) = \left( \begin{array}{ccc}
1 & x & 0 \\
 & 1 & 0 \\
 & & 1 \end{array} \right), \quad 
 u_{\beta_0}(x) = \left( \begin{array}{ccc}
1 & 0 & x \\
 & 1 & 0 \\
 & & 1 \end{array} \right), \quad 
  u_{\beta_0 - \beta}(x) = \left( \begin{array}{ccc}
1 & 0 & 0 \\
 & 1 & x \\
 & & 1 \end{array} \right)  
 \]
 We also let
 \[  w_{\beta} = \left( \begin{array}{ccc}
0& 1 & 0 \\
 -1 & 0 & 0 \\
 0 &0  & 1 \end{array} \right) \]
 denote the Weyl group element associated to $\beta$.
 The short root spaces, on the other hand,  are $e_i \otimes \mathcal{R}$ and $e_i' \otimes \mathcal{R}'$ and are thus identifiable with $\mathcal{R}$.

 \vskip 5pt

 \subsection{\bf Exceptional dual pairs.} 
 We can now exhibit 2 families of  dual pairs in $S_{\mathcal{R}}$. 
 \vskip 5pt
 
 \begin{itemize}
 \item From (\ref{E:lie}), one has
 \[  \mathfrak{sl}_3 \oplus \mathfrak{m}_{\mathcal{R}} \subset  \mathfrak{s}_{\mathcal{R}}. \]
 This gives a family of dual pairs
 \begin{equation} \label{E:dualpair1} 
   SL_3 \times  Aut(\mathcal{R}, \det) \longrightarrow S_{\mathcal{R}}.
   \end{equation}
 
 \vskip 5pt
\item For a pair of Jordan algebras $\mathcal{R}_0 \subset \mathcal{R}$, we have $\mathfrak{s}_{\mathcal{R}_0} \subset \mathfrak{s}_{\mathcal{R}}$ which gives
 a subgroup $G_{\mathcal{R}_0} \subset S_{\mathcal{R}}$, where $G_{\mathcal{R}_0}$ is isogeneous to $S_{\mathcal{R}_0}$.
  If $G'_{\mathcal{R}_0, \mathcal{R}}=   Aut(\mathcal{R}, \mathcal{R}_0)$,
   then one has a second family of  dual pairs
 \begin{equation}  \label{E:dualpair2}
  G_{\mathcal{R}_0} \times G'_{\mathcal{R}_0,\mathcal{R}} \longrightarrow S_{\mathcal{R}}. \end{equation}
 With $\mathcal{R}_0 \subset \mathcal{R}$ fixed, we shall simply write $G \times G'$ for this dual pair. For the various pairs $\mathcal{R}_0 \subset \mathcal{R}$ of interest here, we tabulate the associated dual pairs in the table below.
\vskip 5pt

\begin{center}
\begin{tabular}{|c|c|c|c|}
\hline
$\mathcal{R}_0 \subset \mathcal{R}$  & $\k  \subset J(K)$  & $E \subset J(D)$ & $J(\k) \subset J(D)$  \\
\hline
$G \times G'$ &  $G_2 \times PGL_3$ & $Spin_8 \times SL_2(E)/\Delta \mu_2$ & $F_4 \times PGL_2$ \\
\hline
\end{tabular}
\end{center}
\end{itemize}
\vskip 5pt
  Observe that in the language of Table 3, with $X = H \backslash G$, the dual pairs described above are precisely  $G_X \times G$. 
  \vskip 5pt

\subsection{\bf Heisenberg parabolic.}
The presentation (\ref{E:lie}) also allows one to describe certain parabolic subalgebras of 
$\mathfrak{s}_{\mathcal{R}}$.  If we consider the adjoint action of 
\[  t  = {\rm diag}(1,0,-1) \in \mathfrak{sl}_3 \]
on $\mathfrak{s}$, we obtain a grading $\mathfrak{s} = \oplus_i \mathfrak{s}[i]$ by the eigenvalues of $t$. Then  
\[ \begin{cases} 
 \mathfrak{s}[0] = \mathfrak{t} \oplus \mathfrak{m} \oplus (e_2 \otimes \mathfrak{R}) \oplus (e_2' \otimes \mathcal{R}')  \\
 \mathfrak{s}[1] = \k e_2' \otimes e_1 \oplus (e_1 \otimes \mathcal{R}) \oplus (e_3' \otimes \mathcal{R}' ) \oplus \k e_3' \otimes e_2 \\
 \mathfrak{s}[2] = \k  e_3' \otimes e_1, \end{cases} \]
and $\mathfrak{p} = \oplus_{i \geq 0} \mathfrak{s}[i]$ is a Heisenberg parabolic subalgebra. 
\vskip 5pt

We denote the corresponding Heisenberg parabolic subgroup by $P_S = M_S \cdot N_S$. In particular, its unipotent radical is a Heisenberg group with 1-dimensional center $Z_S \cong u_{\beta_0}(\k) \cong \mathfrak{s}[2]$  and  
\[  N_S/Z_S \cong \mathfrak{s}[1] = \k \oplus  \mathcal{R}  \oplus \mathcal{R}' \oplus \k, \]
 The semisimple  type of its Levi factor $M_S$ is given in the table below. 
 \vskip 5pt
\begin{center}
\begin{tabular}{|c|c|c|c|c|}
\hline 
$S$ & $F_4$ & $E_6$ & $E_7$ & $E_8$ \\ 
\hline
$M_S$ & $C_3$ & $A_5$ &$D_6$ & $E_7$ \\
\hline
\end{tabular}
\end{center}
\vskip 5pt
The Lie bracket defines an alternating form on $N_S/Z_S$ which is fixed by $P_S^1 = [P_S , P_S]$. This gives an embedding 
\[  P_S^1 = M_S^1 \cdot N_s \hookrightarrow Sp(N_S/Z_S) \ltimes N_S. \]
\vskip 5pt

\subsection{\bf Intersection with dual pairs.} 
For a pair $\mathcal{R}_0 \subset \mathcal{R}$, with associated dual pair given in (\ref{E:dualpair2}), it follows by construction that
\[ ( G_{\mathcal{R}_0} \times G'_{\mathcal{R}_0, \mathcal{R}} ) \cap P_S = 
P \times   G'_{\mathcal{R}_0, \mathcal{R}}, \]
where $P$ is the Heisenberg parabolic subgroup of $G_{\mathcal{R}_0}$.
On the other hand, for the family of dual pairs given in (\ref{E:dualpair1}), 
\[ ( SL_3 \times Aut (\mathcal{R},\det)) \cap P_S = B \times  Aut (\mathcal{R},\det) \]
where $B$ is a Borel subgroup of $SL_3$.

\vskip 5pt

\subsection{\bf Siegel parabolic.}
The group $S$ of type $E_6$ or $E_7$
has a Siegel parabolic subgroup $Q_S = L_S \cdot U_S$ whose unipotent radical $U_S$ is abelian; we call this a Siegel parabolic subgroup.  The semisimple type of $L_S$ and the structure of $U_S$ as an $L_S$-module is summarized in the following table.

\vskip 10pt

\begin{center}
\begin{tabular}{|c|c|c|c|}
\hline 
$S$ & $L_S$ & $U_S$  &$U_S$ as $L_S$-module  \\
\hline 
  $E_6$ & $D_5$ & $\mathbb{O} \oplus \mathbb{O}$ & half spin representation of dimension $16$ 
  \\
\hline
$E_7$ & $E_6$ & $J(\mathbb{O})$ & miniscule representation of dimension $27$ \\
\hline
\end{tabular}
\end{center}
\vskip 5pt

\noindent Let $\Omega_Q \subset \overline{U}_S$ be the orbit of a highest weight vector in $\overline{U}_S$. The following proposition describes the set $\Omega_Q$:
\vskip 5pt

\begin{proposition}
(i) If $S$ is of type $E_6$, then 
\[  \Omega_Q = \{ (x,y) \in \mathbb{O}^2:  N(x) = N(y) = 0 = x \cdot \bar{y} \}. \] 
\vskip 5pt

(ii) If $S$ is of type $E_7$, then
\[  \Omega_Q = \{ \alpha \in J: \text{rank}(\alpha) =1 \}. \]
\end{proposition}  
\vskip 5pt

\subsection{\bf Intersection with dual pairs.} 
With $\mathcal{R}_0 \subset \mathcal{R}$ fixed, with associated dual pair $G \times G'$ as given in (\ref{E:dualpair2}), one may choose $Q_S$ so that
\[  (G \times G') \cap Q_S = G \times Q _0\]
with $Q_0 = L_0 \cdot U_0$ a Siegel parabolic subgroup of $G'$, so that $U_0$ is abelian.  The group $Q_0$ and the embedding $U_0 \subset U_S$ can be described by the following table.
\begin{center}
\begin{tabular}{|c|c|c|c|}
\hline
$G'$  &   $PGL_3$ & $SL_2(E)/\Delta \mu_2$ & $PGL_2$ \\
\hline
$Q_0$ & maximal parabolic & Borel & Borel \\
\hline
$U_0 \subset U_S$ & $\k^2 \subset \mathbb{O}^2$ & $E \subset J(\mathbb{O})$ & $\k \subset J(\mathbb{O})$ \\
\hline
\end{tabular}
\end{center} 
Identifying the opposite unipotent radical $\bar{U}_0$ with the dual space of $U_0$ using the Killing form, one has a natural projection 
\[  \tau: \bar{U}_S \longrightarrow U_0. \]
This is simply given by the projection from $U_S$ to $U_0$ along $U_0^{\perp}$.
 


\section{\bf Generic Orbits}
In this section, we consider an orbit problem which will be important for our applications.
Namely, with the notation at the end of the last section, we have an action of $L_0 \times G$
on the set $\Omega_Q \subset \overline{U}_S$. We would like to determine the generic orbits of this action. For simplicity, we shall consider the case when $S = E_6$ and $E_7$ separately.
\vskip 5pt

\subsection{\bf Dual Pair in $E_6$.}
 Suppose first that $S  = E_6$ so that $G' \times G = PGL_3 \times G_2$. In this case, the natural $L \times G_2$-equivariant projection $\tau:  \bar{U}_S \longrightarrow \bar{U} _0$ is given by
\[  \tau(x,y) = (Tr(x), Tr(y)). \]
 The nonzero elements in $\bar{U}_0 \cong \k^2$ are in one orbit of $L_0$; we fix a representative $(0,1) \in \k^2$ and note that its stabilizer in $L_0$ is the ``mirabolic" subgroup $P_{L_0}$ of $L_0 \cong GL_2$. 
Then the fiber over $(0,1)$ is given by
\[  \{ (x,y) \in \mathbb{O}^2:  N(x) = N(y) = Tr(x) = 0, \, Tr(y) =1, \, x \cdot \bar{y} = 0\}, \]
and carries a natural action of $P_{L_0} \times G_2$. We note:

\vskip 5pt

\begin{lemma}  \label{L:G2}
(i) The group $G_2$ acts transitively on the fiber $\tau^{-1}(0,1)$ and the stabilizer of a point $(x_0, y_0)$ is isomorphic to the subgroup $[L,L] \cdot Z(U) \subset J$. 
\vskip 5pt

(ii) If we consider the subset 
$\{  (x_0 , y_0 + \lambda x_0):  \lambda \in \k \} \subset \tau^{-1}(0,1)$, then the subgroup of $P_{L_0} \times G_2$ stabilizing this subset is isomorphic to 
\[  (P_{L_0} \times L \cdot [U,U])^0 = \{ (h, g \cdot u):  \det h = \det g \}. \]
The action of the element
\[  \left( \begin{array}{cc}
a & b \\
0 & 1 \end{array} \right) \times g \cdot u \in (P_{L_0} \times L \cdot [U,U])^0 \]
is by
\[  (x_0, y_0+ \lambda x_0) \mapsto  (x_0,  y_0+ a^{-1} \cdot (\lambda + b - p(u)) x_0) \]
where  $p: J \longrightarrow \k \cong J / [L,L] \cdot Z(U)$ is the natural projection.  
Thus, there is a unique generic $L_0 \times G_2$ orbit on $\Omega_Q$ given by
\[  (L_0 \times G_2) \times_{(P_{L_0} \times L \cdot [U,U])^0}  \k. \]

\end{lemma} 
\vskip 5pt

\subsection{\bf Dual Pairs in $E_7$.}
 Now suppose that $S = E_7$. 
As above,  we first determine the generic $L_0$-orbits on $\bar{U}_0$. For each generic $L_0$-orbit  in $\bar{U}_0$, let us take a representative $\chi$ and let $Z_{\chi}$ denote its stabilizer in $L_0$.  Then the fiber  $\tau^{-1}(\chi)$ is preserved by $Z_{\chi} \times G$. In each case, it follows by Prop. \ref{P:jordan}(ii) and Prop. \ref{P:D4} that $G$ acts transitively on $\tau^{-1}(\chi)$. Denote the stabilizer in $G$ of $\tilde{\chi} \in \tau^{-1}(\chi)$ by $H_{\chi}$.  Then under the action of $Z_{\chi} \times G$,  the stabilizer group $\tilde{H}_{\chi}$ of $\tilde{\chi}$ sits in a short exact sequence
\[  \begin{CD}
1 @>>> H_{\chi} @>>> \tilde{H}_{\chi} @>>p> Z_{\chi} @>>> 1. \end{CD} \]
 In fact, $\tilde{H}_{\chi}$ is a direct product
\[  \tilde{H}_{\chi} \cong \Delta Z_{\chi} \times H_{\chi}\subset Z_{\chi} \times G \]
Thus,  the generic $L_0 \times G$-orbits are given by the disjoint union
\[  \bigcup_{\text{generic $\chi$}}  (Z_{\chi} \times G) \times_{\tilde{H}_{\chi}} \tilde{\chi} \]
where the union runs over the generic $L_0$-orbits on $\bar{U}_0$ and $\tilde{\chi}$ is an element in $\tau^{-1}(\chi)$ with stabilizer $\tilde{H}_{\chi}$. 
We summarize this discussion in the following table. 
\vskip 5pt

\begin{center}
\begin{tabular}{|c|c|c|}
\hline 
$G \times G'$ & $F_4\times PGL_2$ & $Spin_8 \times SL_2(E)/\Delta\mu_2$   \\
\hline
 generic $L_0$-orbits   & singleton  &  $(a,b,c) \in (\k^{\times}/ \k^{\times 2})^3/\Delta \k^{\times}$   \\
 \hline  
 $\tau^{-1}(\chi)$ & $\alpha \in J(\mathbb{O})$ of rank $1$ and  trace $1$  & $(x,y,z) \in \mathbb{O}^3$ with $Tr(xyz) = abc$    \\
\hline
$Z_{\chi}$ & trivial & center of $G' = \mu_2 \times \mu_2$   \\
\hline
$H_{\chi}$ & $Spin_9$ & $G_2$    \\
\hline
 
\end{tabular}
\end{center}
 
\vskip 10pt


\vskip 5pt

\section{\bf Minimal Representation}
In this section, we introduce the (unitary) minimal representation $\Pi$ of $S$ and describe some models for $\Pi$. Note that when $S = F_4$, $\Pi$ is actually a representation of the double cover of $F_4$. When $S$ is of type $E$, then $\Pi$ is a representation of $S$.
\vskip 5pt

\subsection{\bf Schrodinger model.} 
Because the groups $S= E_6$ and $E_7$ have a Siegel parabolic subgroup, there is an analog of the Schrodinger model for the minimal representation $\Pi$ of $S$. 
By \cite{DS}, the representation $\Pi$ can be realized on the space $L^2(\Omega_Q,\mu_Q)$ of square-integrable functions on $Q$ with respect to a $L_S$-equivariant measure  $\mu_Q$ on $\Omega_Q$. This is analogous to the Schrodinger model of the Weil representation. In particular, we have the following action of $Q_S$ on $\Pi$:
\[  \begin{cases}
(l\cdot f)(\chi)  = \delta_{Q_S}(l)^r \cdot f(l^{-1} \cdot \chi) \\
(u \cdot f)(\chi) = \chi(u)  \cdot f(\chi),
\end{cases} \]
where $r = 1/4$ (resp. $2/9$) if $S$ is of type $E_6$ (resp. $E_7$). 
\vskip 5pt

 \subsection{\bf Mixed model.}
For general $S = S_{\mathcal{R}}$, one has the analog of the mixed model, on which the action of the Heisenberg group $P_S$ is quite transparent.  
Recall that $N_S/Z_S = \k \oplus \mathcal{R} \oplus \mathcal{R}' \oplus \k$ and one has an embedding 
\[  P_S^1 = [P_S, P_S] \hookrightarrow {\rm Sp}(N_S/Z_S) \ltimes N_S.\]
 Then by \cite{KS},  the mixed model of the minimal representation is realized on  the Hilbert space
\[  {\rm Ind}_{P_S^1}^{P_S} L^2(\mathcal{R}' \oplus \k') \cong L^2(\k^{\times} \oplus \mathcal{R} \oplus \k), \]
where the action of $P_S^1$  on $L^2(\mathcal{R} \oplus \k)$ is via the Heisenberg-Weil representation (associated to any fixed additive character $\psi$ of $\k$). The explicit formula can be found in \cite[Prop. 43]{R}. 
\vskip 5pt

In fact, one can describe the full action of $S$ on $\Pi$ by giving the action of an extra Weyl group element. More precisely, if $w_{\beta}$ is the standard Weyl group element in $SL_3$ associated to the root $\beta$ (see \S 5.1), then by \cite[Prop. 47]{R}, one has
\[   (w_{\beta} \cdot f)(t, x, a) = \psi(\det(x)/a) \cdot  f(-a/t, x, -a). \]
Since $S$ is generated by $P_S$ and the element $w_0$, this completely determines the representation $\Pi$. 
\vskip 5pt

For example, one may work out the action of an element $u_{-\beta}(b) = w_{\beta} u_{\beta}(b) w_{\beta}^{-1}$ (see \S 5.1).  A short computation gives:
\[  (u_{-\beta}(b) \cdot f)(t,x,a) = \psi\left(\frac{b \det(x)}{a-t^2}\right)  \cdot  f(t - \frac{ab}{t}, a - \frac{a^2b}{t^2}, x). \]
If $f$ is continuous, then the above formula gives:
\begin{equation} \label{E:weyl}
  (u_{-\beta}(b) \cdot f)(1,x,0)  = \psi( -b \det(x)) \cdot f(1,x,0). \end{equation}
This formula will be useful in the last section.

 \vskip 10pt

\section{\bf Exceptional Theta Correspondences: $G \times G'$}
Now we may study the restriction of the minimal representation $\Pi$ to the dual pairs introduced earlier. In this section, we shall treat the family of dual pairs $G \times G'$ given in (\ref{E:dualpair2}).  For simplicity, we shall consider the case when $S = E_6$ and $E_7$ separately. 
\vskip 5pt

\subsection{\bf Restriction to $G \times G' \subset E_7$.}
Suppose first that $S$ is of type $E_7$, so that $\Omega_Q$ is the set of rank $1$ elements in $J= J(\mathbb{O})$. Consider the Schrodinger model for $\Pi$.  On restricting $\Pi$ to $Q _0 \times G$, we have the following formulae:
\[
\begin{cases} 
(g \cdot f)(\alpha) = f( g^{-1} \cdot \alpha) \quad \text{for $g \in G$;} \\
(u(a) \cdot f)(\alpha) =  \psi( tr(a \cdot \alpha))  \cdot f(\alpha) \quad \text{for $u(a) \in U_0$;} \\
(l \cdot f)(\alpha) =  |\det (l)|^s \cdot f(l^{-1} \cdot \alpha) \quad \text{for $l \in L_0$,}
\end{cases}  \]
where $s$ is a real number whose precise value will not be important to us here. 
\vskip 5pt

From our description of generic $L_0 \times G$-orbits given in \S 6.2, we deduce as in the derivation of (\ref{eq:L2Xdexplicit}) that as a $Q_0\times G$-module, 
\begin{equation} \label{E:explicit}
  \Pi  \cong \bigoplus_{\text{$\chi$ generic} }
{\rm Ind}^{Q_0 \times G}_{U_0 \times \tilde{H}_{\chi}} \chi \boxtimes 1 
\cong   \bigoplus_{\text{$\chi$ generic} } {\rm Ind}^{Q_0}_{Z_{\chi} \cdot U_0} L^2(H_{\chi} \backslash G).
\end{equation}
Here, $G$ and $Z_{\chi}$ act on $L^2(H_{\chi} \backslash G)$ by right and left translation
respectively, and $U_0$ acts by $\chi$. 
 
 
\vskip 10pt

\subsection{\bf Abstract decomposition.}
On the other hand,  there is an abstract direct integral decomposition
\[  \Pi = \int_{\widehat{G'}} \pi \boxtimes \Theta(\pi) \, d\nu_{\Theta}(\pi). \] 
 Restricting to $Q_0$, we may write: 
 \[  \pi|_{Q_0}  \cong \bigoplus_{\chi}  {\rm Ind}^{Q_0}_{Z_{\chi} \cdot U_0}   W_{\chi}(\pi) \]
 for some $Z_{\chi} \cdot U_0$-module $W_{\chi}(\pi)$ with $U_0$ acting via $\chi$. Thus, 
 \begin{equation} \label{E:abstract}  
 \Pi \cong \bigoplus_{\chi}  \int_{\widehat{G'}}  {\rm Ind}^{Q_0}_{Z_{\chi} \cdot U_0} W_{\chi}(\pi) \boxtimes \Theta(\pi) \,  d\nu_{\Theta}(\pi). 
  \end{equation}
\vskip 5pt

\subsection{\bf Comparison.}
Comparing (\ref{E:explicit}) and (\ref{E:abstract}), we deduce that
 there is  an isomorphism of $G_A$-modules:
 \begin{equation} \label{E:exceptional}
   L^2(H_{\chi} \backslash G) \cong \int_{\widehat{G'}} W_{\chi}(\pi) \boxtimes \Theta(\pi) \, d\nu_{\Theta}(\pi). \end{equation}
 Since $G'$ is isogenous to a product of $SL_2$, the space $W_{\chi}(\pi) =Wh_{\chi}(\pi)$ has been determined in Theorem \ref{T:bessel}(3) and is at most 1-dimensional.  \vskip 10pt

\subsection{\bf Mixed model.}
 To explicate  the measure $d\nu_{\Theta}(\pi)$,  we consider the mixed model of $\Pi$ restricted to $P \times G'$. Since
 \[  N /Z_S = \k \oplus \mathcal{R}_0 \oplus \mathcal{R}_0' \oplus \k \subset N_S/Z_S. \]
 Under its adjoint action on $\mathcal{R} \oplus \k$, 
  $G'$ fixes $\mathcal{R}_0 \oplus \k$ pointwise, and its action on $\mathcal{R}_0^{\perp}$ is  described in the following table.
 \vskip 5pt
 
 \begin{center}
\begin{tabular}{|c|c|c|}
\hline 
$G'$ & $\mathcal{R}_0$  & $\mathcal{R}_0^{\perp}$   \\
\hline
  $PGL_2$ & $J(\k)$ & $adjoint^{\oplus 3}$ \\
\hline
 $SL_2^3/\Delta \mu_2$ & $\k^3$ &  $\oplus_{i=1}^3 std_i \boxtimes std_{i+1}^{\vee}$ \\
\hline
\end{tabular}
\end{center}
 \vskip 5pt
 
\noindent Thus as a representation of $G'$, we have:
 \[
   \Pi 
   \cong L^2(\k^{\times}) \otimes   L^2(\mathcal{R}_0 \oplus \k) \otimes L^2(\mathcal{R}_0^{\perp}) \]
   where $G'$ acts only on $L^2(\mathcal{R}_0^{\perp})$ and the action is geometric. Thus, $\Pi$ is weakly equivalent to $L^2(\mathcal{R}_0^{\perp})$ as a representation of $G'$.
    By our description of the $G'$-module $\mathcal{R}_0^{\perp}$, we have:
  \begin{lemma}
 The representation $L^2(\mathcal{R}_0^{\perp})$ (and hence $\Pi$) is weakly equivalent to the regular representation $L^2(G')$
    \end{lemma}
 \begin{proof}
 When $G' = PGL_2$, this follows from Corollary \ref{C:rank1}.
 When $G' = SL_2^3/\Delta \mu_2$, the representation of $G'$ on $E_A^{\perp}$ is the restriction of a representation of $\tilde{G'} = GL_2^3 / \Delta \k^{\times}$ (by the same formula). Now the action of $\tilde{G'}$ on $\mathcal{R}_0^{\perp}$ has finitely many open orbits with representatives $(1,1, g) \in GL_2^3$ with $g$ regular semisimple, and the stabilizer of such a representative is $\Delta T$ with $T$ a maximal torus in $PGL_2$. Hence, as a representation of $\tilde{G'}$, 
 $L^2(\mathcal{R}_0^{\perp})$ is weakly equivalent to
 \[  \bigoplus_T  {\rm Ind}^{\tilde{G'}}_{\Delta T} \C 
 \cong \bigoplus_T {\rm Ind}^{\tilde{G'}}_{\Delta PGL_2}   L^2(T \backslash PGL_2)  \]
 as $T$ runs over conjugacy classes of maximal tori in $PGL_2$.
  By Corollary \ref{C:rank1} and the continuity of induction, we deduce that $L^2(\mathcal{R}_0^{\perp})$ is 
 weakly equivalent to $L^2(\tilde{G'})$. Thus, on restriction to $G'$, $L^2(E^{\perp})$ is weakly equivalent to $L^2(G')$, as desired.
 \end{proof}
 \vskip 5pt
 Concluding, we have:

  \begin{theorem}
 There is  an isomorphism of $G$-modules:
 \[  L^2(H_{\chi} \backslash G) \cong \int_{\widehat{G'}} W_{\chi}(\pi) \boxtimes \Theta(\pi) \, d\mu_{G'}(\pi), \]
   with $W_{\chi}(\pi) = Wh_{\chi}(\pi)$ as given in Theorem \ref{T:bessel}(3)  and  $\mu_{G'}$
 is  the Plancherel measure.
\end{theorem}
   \vskip 5pt
In addition,   as we discussed in \S 3.7, the smoooth analog of our argument in this section implies that
   \[  W_{\chi}(\pi) = Wh_{\chi}(\pi) \cong  Hom_{H_{\chi}}(\Theta^{\infty}(\pi^{\infty}), \C) = 
    Hom_{H_{\chi}}(\Theta(\pi)^{\infty}, \C). \]
     \vskip 10pt
     
\subsection{\bf Restriction to $PGL_3 \times G_2$.}
We now treat the dual pair $PGL_3 \times G_2$ in $S = E_6$, which 
 can be done by a similar analysis. In this case, $\Omega_Q \subset \mathbb{O}^2$.  If we restrict the action of $S$ to $Q_0 \times G_2$, 
we deduce by Lemma \ref{L:G2}(ii) that as a representation of $Q \times G_2$, 
\[  \Pi  \cong   {\rm Ind}^{Q_0 \times G_2}_{(P_{L_0} \times L \cdot [U,U])^0 \cdot U}  L^2(\k) \]
where the action of $(P_{L_0} \times L\cdot [U,U])^0$ on $L^2(\k)$ is given through the geometric action described in Lemma \ref{L:G2}(ii) and the action of $U_0$ is by a nontrivial character fixed by $P_{L_0}$.
 \vskip 5pt
 
 By using the Fourier transform on $L^2(\k)$, we deduce that as a representation of 
$(P_{L_0} \times L\cdot [U,U])^0$, 
\[  L^2(\k) \cong {\rm Ind}_{U_{L_0} \times J}^{(P_{L_0} \times L\cdot [U,U])^0} \psi^{-1} \boxtimes \psi.\]
Hence, as a representation of $Q_0 \times G_2$
\begin{equation} \label{E:G2}   \Pi \cong {\rm Ind}^{Q_0}_{N_0} \chi \boxtimes {\rm Ind}^{G_2}_J \psi  \end{equation}
where $N_0 = U_{L_0} \cdot U_0$ is the unipotent radical of a Borel subgroup of $PGL_3$ and $\chi$ is a generic character of $N_0$. 
\vskip 5pt

On the other hand, we have abstractly
\begin{equation} \label{E:G22}   
\Pi \cong \int_{\widehat{PGL_3}} \pi|_{Q_0} \otimes \Theta(\pi) \, d \nu_{\Theta}(\pi). \end{equation}
We note that if $\pi$ is tempered, then 
\[  \pi|_{Q_0} \cong {\rm Ind}_{N_0}^{Q_0} \chi, \]
in which case we deduce on comparing (\ref{E:G2}) and (\ref{E:G22}) that
\begin{equation} \label{E:G23}
   L^2((J,\psi) \backslash G_2) = {\rm Ind}_J^{G_2} \psi \cong 
\int_{\widehat{PGL_3}}  \Theta(\pi) \, d\nu_{\Theta}(\pi). \end{equation}
For (\ref{E:G23}) to hold, we thus need to show that $\nu_{\Theta}$ is absolutely continuous with respect to the Plancherel measure of $PGL_3$.
\vskip 5pt

For this, we examine the mixed model of $\Pi$ which  is realized on $L^2(\k^{\times} \times J(\k^2) \times \k)$. Noting that $J(\k^2) \cong \mathfrak{gl}_3$ as $PGL_3$-module \cite{MS}, we deduce  that as a representation of $PGL_3$, $\Pi$ is weakly equivalent to the representation on $L^2(\mathfrak{sl}_3)$ associated to the adjoint action on $\mathfrak{sl}_3$. 
As in Corollary \ref{C:rank1}, we know that $L^2(\mathfrak{sl}_3)$ is weakly equivalent to  
$\bigoplus_T L^2(T \backslash PGL_3)$,  with $T$ running over conjugacy classes of maximal tori in $PGL_3$. 
\vskip 5pt

Using the same argument 
as in \cite[\S 6]{Sakellaridis-Venkatesh}, one can show that for each $T$, the spectral measure for $L^2(T \backslash PGL_3)$ is absolutely continuous with respect to the Plancherel measure of $PGL_3$, and hence so is the spectral measure of $L^2(\mathfrak{sl}_3)$;  this justifies (\ref{E:G23}) and shows that
  \[  L^2((J,\psi) \backslash G_2) = {\rm Ind}_J^{G_2} \psi \cong 
\int_{\widehat{PGL_3}}  W(\pi) \otimes \Theta(\pi) \, d\mu_{PGL_3}(\pi) \]
  for some multiplicity space $W(\pi)$ of dimension $\leq 1$.
  \vskip 5pt
  
It is natural to state:

\vskip 5pt

\begin{conjecture}
For an adjoint simple algebraic group $G$, 
the representation   $L^2(\mathfrak{g})$ of $G$ is weakly equivalent to the regular representation $L^2(G)$. 
\end{conjecture}
\vskip 5pt

Corollary \ref{C:rank1} verifies this conjecture for $PGL_2$.
If the conjecture holds for $PGL_3$, one can then take $W(\pi)$ to be $\C$ for all $\pi$.

\vskip 5pt

\section{\bf Exceptional Theta Correspondence:  $SL_3 \times Aut(\mathcal{R},\det)$}

Finally we come to the family of dual pairs  $SL_3 \times Aut(\mathcal{R},\det) \subset S = S_{\mathcal{R}}$ given by (\ref{E:dualpair1}). What is interesting about this situation is that the group $S$ may have no Siegel parabolic subgroup, so that the argument below is not the analog of that in the classical cases of \S 3.
To simplify notation, we shall set $G =  Aut(\mathcal{R},\det)$. Note that in the case of $F_4$, $S$ is the double cover of $F_4$ and the dual pair is $\tilde{SL}_3 \times G = \tilde{SL}_3 \times SL_3$.  
\vskip 5pt

Let $Q_0 = L_0 \cdot U_0 \subset SL_3$ be the maximal parabolic subgroup stabilizing the subspace $\k e_1 + \k e_2$, so that  
\[  L_0 \cong GL_2 \quad \text{and} \quad
U_0 = u_{\beta_0 - \beta}(\k) \times u_{\beta_0}(\k).\]
  Let $\chi$ be a generic character of $U_0$ trivial on $u_{\beta_0 - \beta}(\k)$.  
The stabilizer in $L_0$ of $\chi$ is a  subgroup of the form $T_0 \ltimes U_{L_0}$  with $T_0 \cong \k^{\times}$ contained in the diagonal torus and $U_{L_0}  = u_{-\beta}(\k)$.  On restricting the minimal representation $\Pi$ to $Q_0 \times G$, we may write
\[  \Pi \cong {\rm Ind}^{Q_0 \times G}_{P_{L_0}U_0  \times G} \Pi_{\chi} \]
for some representation $\Pi_{\chi}$ of $P_{L_0} U_ 0 \times G$ with $U_0$ acting by $\chi$. 
Here, we have used the theorem of Howe-Moore which ensures that the trivial character of $U_0$ does not intervene.
\vskip 5pt

 Now we can describe the $P_{L_0} U_0\times G$-module $\Pi_{\chi}$ using the mixed model of $\Pi$.  Recall that this mixed model of $\Pi$ is realized on $L^2(\k^{\times} \times \mathcal{R} \times \k)$. Moreover, the action of $U_0  = u_{\beta_0 - \beta}(\k) \times u_{\beta_0}(\k)$ in this model is:
 \[  \begin{cases}
 (u_{\beta_0}(z)f )(t, x, a) = \psi(tz) \cdot f(t,x,a) \\
(u_{\beta_0 - \beta}(y) f)(t,x,a) = \psi(ay) \cdot f(t,x,a). \end{cases} \] 
As such,  $\Pi_{\chi}$ is the representation obtained from $\Pi$ by specializing (continuous) functions $f \in \Pi$  to the function $x \mapsto f(1, x, 0)$ of $\mathcal{R}$. Thus 
\[  \Pi_{\chi} = L^2(\mathcal{R})  \]
where the action of $T_0 \times G$ is geometric, with $T_0$ acting by scaling on $\mathcal{R}$. Moreover, it follows by  (\ref{E:weyl}) that the action of $u_{-\beta}(b) \in U_{L_0}$ is:
\[  (u_{-\beta}(b) \cdot f)(x)  = \psi (- b \cdot \det(x)) \cdot f(x). \]
\vskip 5pt

Now the set $\{ x \in \mathcal{R}: \det(x) \ne 0\}$ is open dense and by Proposition \ref{P:jordan}(i), it  is the union of finitely many generic orbits  of $T_0 \times G$ indexed by $\k^{\times}/ (\k^{\times})^3$. For each $a \in \k^{\times}/ (\k^{\times})^3$, let $H_a$ be the corresponding stabilizer group whose type is described in Table 4 in \S 4.2. Then
\[  \Pi \cong \bigoplus_a {\rm Ind}^{Q_0 \times G}_{N_0 \times H_a}   \chi_a \boxtimes \C \cong {\rm Ind}_{N_0}^{Q_0} \chi_a \boxtimes L^2(H_a \backslash G). \]
On the other hand, one has abstractly
\[  \Pi \cong \int_{\widehat{SL_3}} \pi|_{Q_0} \otimes \Theta(\pi) \,  d\nu_{\theta}(\pi). \]
Now we note:
\vskip 5pt

\begin{lemma}
As a representation of $SL_3$, $\Pi$ is weakly equivalent to $L^2(SL_3)$.
\end{lemma}

\begin{proof}
If $S$ is of type $E$, the group $SL_3$ is contained in a conjugate of the Heisenberg parabolic subgroup $P_S$. Indeed, after an appropriate  conjugation, we may assume that 
\[
SL_3 \subset Aut(J(\k^2), \det) = SL_3 \times_{\mu_3} SL_3 \subset \Aut(J(B), \det), \]
where $B = \k^2$, $M_2(\k)$ or the split octonion algebra $\mathbb{O}$ in the respective case.
From the description of the mixed model, one sees that $\Pi$ is nearly equivalent to the representation of $SL_3$ on $L^2(J(B))  = L^2(J(\k^2)) \otimes L^2(J(\k^2)^{\perp})$.  Since $J(\k^2) \cong M_3(\k)$ 
with $SL_3$ acting by left multiplication, we see that $J(\k^2)$ is weakly equivalent to the regular representation of $SL_3$. This implies that
  $\Pi$ is weakly equivalent to   the regular representation of $SL_3$. 
\vskip 5pt

The case when $S = F_4$ is a bit more intricate; we omit the details here.
\end{proof} 

Thus $\nu_{\theta} = \mu_{SL_3}$ and  every $\pi$ in the support of $\nu_{\theta}$ is tempered, so that 
\[  \pi|_{Q_0} = \bigoplus_{a\in \k^{\times}/ (\k^{\times})^3}   Wh_{\chi_a}(\pi) \otimes {\rm Ind}_{N_0}^{Q_0} \chi_a. \]
Comparing, we see that
\[  L^2(H_a \backslash G) \cong  \int_{\widehat{SL_3}} Wh_{\chi_a}(\pi) \otimes \Theta(\pi)  d\mu_{SL_3}(\pi),\]
as desired.

\end{document}